\documentclass[uplatex,11pt,a4paper]{amsart} % arXiv用
\usepackage{mathtools,amssymb,amsthm}
\usepackage{bm}
\usepackage{mathrsfs} % \mathscr
\usepackage{stmaryrd} % inverse \mapsto (\mapsfrom and \longmapsfrom) https://ftp.yz.yamagata-u.ac.jp/pub/CTAN/fonts/stmaryrd/stmaryrd.pdf
\usepackage{eucal}

\usepackage{xcolor, graphicx, here}
\definecolor{darkgreen}{rgb}{0.0, 0.6, 0.0}

\usepackage{tikz}
\usetikzlibrary{intersections, calc, arrows, cd}

\usepackage[colorlinks,citecolor=darkgreen,linkcolor=red]{hyperref} % arXiv用

\usepackage{geometry}
\numberwithin{equation}{section}
\numberwithin{figure}{section}
\numberwithin{table}{section}
\allowdisplaybreaks %align環境で式の途中で改ページを許す

\usepackage{enumitem}
\setenumerate[1]{label={\upshape(\arabic*)},nosep}
\setenumerate[2]{label={\upshape(\alph*)},nosep}
\setitemize{nosep}
\setdescription{nosep}

\newlist{enumarabic}{enumerate}{1}
\setlist*[enumarabic]{label={\upshape(\arabic*)}, nosep}

\newlist{enumAlph}{enumerate}{1}
\setlist*[enumAlph]{label={\upshape(\Alph*)}, nosep}

\newlist{enumalph}{enumerate}{1}
\setlist*[enumalph]{label={\upshape(\alph*)}, nosep}

\newlist{enumRoman}{enumerate}{1}
\setlist*[enumRoman]{label={\upshape(\Roman*)}), nosep}

\newlist{enumroman}{enumerate}{1}
\setlist*[enumroman]{label={\upshape(\roman*)}, nosep}

\usepackage{cleveref}
 \crefname{section}{\S\!}{\S\S\!}
 \crefname{subsection}{\S\!}{\S\S\!}
 \crefname{subsubsection}{\S\!}{\S\S\!}

\theoremstyle{plain}
\newtheorem{theorem}{Theorem}[section]
 \crefname{theorem}{Theorem}{Theorems}
\newtheorem{theorem-definition}[theorem]{Theorem-Definition}
 \crefname{theorem-definition}{Theorem-Definition}{Theorem-Definition}
\newtheorem{lemma}[theorem]{Lemma}
 \crefname{lemma}{Lemma}{Lemmas}
\newtheorem{proposition}[theorem]{Proposition}
 \crefname{proposition}{Proposition}{Propositions}
\newtheorem{proposition-definition}[theorem]{Proposition-Definition}
 \crefname{proposition-definition}{Proposition-Definition}{Proposition-Definition}
\newtheorem{corollary}[theorem]{Corollary}
 \crefname{corollary}{Corollary}{Corollaries}

 \crefname{fact}{Fact}{Facts}
\newtheorem{claim}[theorem]{Claim}
 \crefname{claim}{Claim}{Claims}

 \crefname{conjecture}{Conjecture}{Conjectures}

 \crefname{problem}{Problem}{Problems}

 \crefname{question}{Question}{Questions}
\newtheorem{theoremA}{Theorem}
 \crefname{theoremA}{Theorem}{Theorems} 

\theoremstyle{definition}
\newtheorem{definition}[theorem]{Definition}
 \crefname{definition}{Definition}{Definitions}
\newtheorem{notation}[theorem]{Notation}
 \crefname{notation}{Notation}{Notations}
\newtheorem{remark}[theorem]{Remark}
 \crefname{remark}{Remark}{Remarks}
 \newtheorem{example}[theorem]{Example}
 \crefname{example}{Example}{Examples}

 \crefname{observation}{Observation}{Observations}

 \crefname{construction}{Construction}{Constructions}

\newenvironment{pfclaim}[1][$\Diamond$]
{\def\claimQED{{#1}}\noindent {\em Proof of the claim. }}
{\leavevmode\unskip\penalty9999 \hbox{}\nobreak\hfill
    \quad\hbox{\claimQED}{\smallskip}}

\newcommand{\ol}{\overline}

\newcommand{\wh}{\widehat}
\newcommand{\wt}{\widetilde}

\newcommand{\xr}{\xrightarrow}

\newcommand{\F}{\mathbb{F}}
\newcommand{\N}{\mathbb{N}}
\newcommand{\Q}{\mathbb{Q}}
\newcommand{\R}{\mathbb{R}}
\newcommand{\Z}{\mathbb{Z}}

\newcommand{\bbA}{\mathbb{A}}
\newcommand{\bbL}{\mathbb{L}}

\newcommand{\bbR}{\mathbb{R}}

\newcommand{\fraka}{\mathfrak{a}}
\newcommand{\frakb}{\mathfrak{b}}

\newcommand{\m}{\mathfrak{m}}
\newcommand{\n}{\mathfrak{n}}
\newcommand{\p}{\mathfrak{p}}

\DeclareMathOperator{\pd}{\mathrm{pd}}

\DeclareMathOperator{\Spec}{Spec}

\DeclareMathOperator{\depth}{depth}
\DeclareMathOperator{\edim}{edim}

\DeclareMathOperator{\rk}{rk}

\newcommand{\torsion}[1]{\textrm{$#1$-}\mathrm{tor}}
\newcommand{\Itor}{\torsion{I}}
\newcommand{\Jtor}{\torsion{J}}

\newcommand{\Iztor}{\torsion{I_0}}
\newcommand{\Izstor}{\torsion{I_0^{s.\flat}}}

\def\et{\mathrm{\acute{e}t}}

\newcommand{\red}{\mathrm{red}}

\newcommand{\calD}{\mathcal{D}}

\newcommand{\Ab}{\mathsf{Ab}}

\newcommand{\id}{\mathrm{id}}
\DeclareMathOperator{\Hom}{Hom}

\DeclareMathOperator{\Tor}{Tor}
\DeclareMathOperator{\Ker}{Ker}

\DeclareMathOperator{\Image}{Im} \renewcommand{\Im}{\Image}
\DeclareMathOperator{\cosk}{cosk}
\DeclareMathOperator{\sk}{sk}

\newenvironment{bsmatrix}{\left[\begin{smallmatrix}}{\end{smallmatrix}\right]}

\begin{document}
\title[Structural properties and tilting correspondences of perfectoid towers]{Structural properties and tilting correspondences of perfectoid towers}

\author[K.\ Hayashi]{Kazuki Hayashi}
\address{Department of Mathematics, Institute of Science Tokyo, 2-12-1 Ookayama, Meguro, Tokyo 152-8551}
\email{hazuki0694@gmail.com}


\begin{abstract}
We prove that every perfectoid tower can be realized as the fiber product of a diagram involving perfectoid towers that are either $p$-torsion free or perfect of characteristic $p$. As an application, we conclude that separated perfectoid towers are reduced. We also establish the tilting invariance of Koszul homology for perfectoid towers. As further applications, we prove that tilting preserves fundamental properties of Noetherian local rings such as being Cohen--Macaulay, Gorenstein, complete intersection, or regular.
\end{abstract}

\maketitle

\setcounter{tocdepth}{2}
\tableofcontents
\section{Introduction}

In commutative algebra in mixed characteristic, the theory of perfectoid rings has become a powerful tool, notably through Y.~Andr\'{e}'s breakthrough resolution of the direct summand conjecture. Recently, S.~Ishiro, K.~Nakazato and K.~Shimomoto \cite{INS25} introduced the notion of \emph{perfectoid towers}, which provide an approximation of perfectoid rings by Noetherian rings, as well as a generalization of perfectoid rings. After that, various constructions of perfectoid towers were found (\cite{Ish26, IS25, HIS26}).
However, their intrinsic properties have remained largely unexplored. In this paper, we therefore study the structure and properties of perfectoid towers and apply them to a fundamental problem in perfectoid theory, namely, \emph{tilting correspondences}.

Let us first review some of results on structural properties of perfectoid rings, mainly from \cite[\S2.1]{CS24}. All typical examples of perfectoid rings are either $p$-torsion free or entirely $p$-torsion (i.e., have characteristic $p$). In fact, any perfectoid ring is decomposed into perfectoid rings with either one of these properties:

\begin{proposition}[{\cite[Rem.\ 8.9]{Lau18}, \cite[\S2.1.3]{CS24}, \cite[16.4.24]{GR24}}]
\label{prop:decompring}
Let $A$ be a perfectoid $\Z_p$-algebra.
Then $\wt{A}\coloneqq A/A_{\textrm{$p$-}\mathrm{tor}}$, $(A/pA)_{\mathrm{red}}$, $(\wt{A}/p\wt{A})_{\mathrm{red}}$ are perfectoid, and we have the decompositions
\[
A\xr{\cong} \wt{A} \times_{(\wt{A}/p\wt{A})_\red} (A/pA)_\red,\quad A^\flat \xr{\cong} (\wt{A})^\flat \times_{(\wt{A}/p\wt{A})_\red} (A/pA)_\red.
\]
\end{proposition}

The proposition reduces many situations involving perfectoid rings to the $p$-torsion free case. For example, we deduce from \cref{prop:decompring} that every perfectoid ring is reduced (\cite[Lemma 8.10]{Lau18}). The decomposition as in \cref{prop:decompring} admits the following converse:

\begin{proposition}[{\cite[Ex.\ 3.8 (6)]{BIM19}, \cite[Prop.\ 2.1.4]{CS24}, \cite[Prop.\ 16.3.25]{GR24}}]
\label{prop:gluingring}
If $B$ is a perfectoid $\Z_p$-algebra and $A'\to(B/pB)_\red$ is any map of perfect $\F_p$-algebras, then $A\coloneqq B\times_{(B/pB)_\red}A'$ is a perfectoid ring.
\end{proposition}

Using \cref{prop:decompring,prop:gluingring}, one can show that the $p$-adic completion of every ind-\'{e}tale algebra over a perfectoid ring is perfectoid (\cite[Corollary 2.2.6]{CS24}).

The first topic in this paper concerns analogues of these results for perfectoid \emph{towers} (\cref{s:str}). In the rest of this section, we fix a perfectoid tower
\begin{equation}
\label{eq:perfectoidtower}
\textrm{{\boldmath $R$}} = \{R_i\}_{i\geq 0} : \quad R=R_0\to R_1\to \cdots \to R_i \to \cdots
\end{equation}
arising from the pair $(R,I_0)$ for a commutative ring $R$ and a principal ideal $I_0\subset R$ (\cref{def:perfectoidtower}). Then its tilt (\cref{def:tilt})
\begin{equation}
\label{eq:tilt}
\textrm{{\boldmath $R$}}^\flat = \{R_i^{s.\flat}\}_{i\geq 0} : \quad R^{s.\flat}=R_0^{s.\flat}\to R_1^{s.\flat} \to \cdots \to R_i^{s.\flat}\to \cdots,
\end{equation}
is a perfectoid tower arising from $(R^{s.\flat},I_0^{s.\flat})$, where $I_0^{s.\flat}$ is the counter-part of $I_0$; for example, one has an isomorphism
\begin{equation}
\label{eq:tiltisom}
R^{s.\flat}/I_0^{s.\flat}\xr{\cong}R/I_0
\end{equation}
(\cref{thm:INS3.35}). Our decomposition theorem of perfectoid towers is stated as follows.

\begin{theoremA}[{\cref{thm:decomp}}]
\label{thm:A}
Consider a perfectoid tower \eqref{eq:perfectoidtower} and its tilt \eqref{eq:tilt}. Then the induced towers $\wt{\textrm{{\boldmath $R$}}} =\{\wt{R_i}\coloneqq R_i/(R_i)_{\textrm{$I_0$-}\mathrm{tor}}\}_{i\geq 0}$, $(\textrm{{\boldmath $R$}}/I_0\textrm{{\boldmath $R$}})_{\mathrm{red}}=\{(R_i/I_0R_i)_{\mathrm{red}}\}_{i\geq 0}$, $(\wt{\textrm{{\boldmath $R$}}}/I_0\wt{\textrm{{\boldmath $R$}}})_{\mathrm{red}}=\{(\wt{R_i}/I_0\wt{R_i})_{\mathrm{red}}\}_{i\geq 0}$ are perfectoid towers, and we have the decompositions
\[
\textrm{{\boldmath $R$}} \xr{\cong} \wt{\textrm{{\boldmath $R$}}}\times_{
(\wt{\textrm{{\boldmath $R$}}}/I_0\wt{\textrm{{\boldmath $R$}}})_{\mathrm{red}}}(\textrm{{\boldmath $R$}}/I_0\textrm{{\boldmath $R$}})_{\mathrm{red}},\quad
\textrm{{\boldmath $R$}}^\flat \xr{\cong} (\wt{\textrm{{\boldmath $R$}}})^\flat \times_{
(\wt{\textrm{{\boldmath $R$}}}/I_0\wt{\textrm{{\boldmath $R$}}})_{\mathrm{red}}}(\textrm{{\boldmath $R$}}/I_0\textrm{{\boldmath $R$}})_{\mathrm{red}}.
\]
\end{theoremA}

One of the core ingredients in its proof is to verify certain diagrams of (possibly) non-unital rings are cartesian, which is achieved by reduction to the positive characteristic case; we also show that, conversely, such cartesian diagrams characterize perfectoid towers (\cref{thm:PB}). As in the case of perfectoid rings, we deduce from \cref{thm:A} that every ring appearing in a perfectoid tower is reduced under the separatedness assumption (\cref{cor:reduced}). Moreover, we establish a gluing theorem for perfectoid towers (\cref{prop:gluetower}).
This result not only provides a tower-theoretic analogue of \cref{prop:gluingring}, but also yields a new method for constructing perfectoid towers with possible $p$-torsion elements.
Prior to this work, the only known construction of this kind was due to Ishiro--Shimomoto (\cite[\S4.4]{IS25}), whose examples are recovered as special cases of our construction (\cref{ex:PB}). 
In contrast to the explicit and laborious calculations in \cite{IS25}, our approach is more conceptual and structural. Furthermore, combining the gluing theorem with \cref{thm:A}, we prove that perfectoid towers are stable under ind-\'{e}tale (or, more adequately, \emph{weakly \'{e}tale}) base change (\cref{thm:bc}).

The second topic, discussed in \cref{s:tilt}, concerns tilting correspondences of perfectoid towers.
We first compare Koszul homology. If $A$ is a perfectoid $\Z_p$-algebra, then the tilt $A^\flat$ admits a counter-part $p^\flat$ of $p$; we have isomorphisms $A^\flat/p^\flat A^\flat\xr{\cong}A/pA$ and $(A^\flat)_{\textrm{$p^\flat$-}\mathrm{tor}}\xr{\cong}A_{\textrm{$p$-}\mathrm{tor}}$. Since these isomorphisms are compatible with each other, one can reformulate it as the isomorphism of Kosuzl complexes $K_\bullet(p;A)\cong K_\bullet(p^\flat;A^\flat)$ (or, of derived quotients $A/^\bbL p\cong A^\flat/^\bbL p^\flat$; see \cite[Lemma 2.10]{Ito23}). A generalization of this fact to perfectoid towers was established in \cite[Remark 3.40]{INS25}. In this context, we apply \cref{thm:A} to obtain the following tilting correspondence of Koszul homology for more general sequences:

\begin{theoremA}[{\cref{thm:tilt-Kos}}]
\label{thm:C}
Consider a perfectoid tower \eqref{eq:perfectoidtower} and its tilt \eqref{eq:tilt}.
Let $\bm{x}=x_1,\ldots,x_n$ be a sequence of elements in $R$ such that $I_0=(x_1)$, and $\bm{x}^{s.\flat}=x_1^{s.\flat},\ldots,x_n^{s.\flat}$ a corresponding sequence in $R^{s.\flat}$ via the isomorphism \eqref{eq:tiltisom}. Then we have an isomorphism of Koszul complexes
\[
K_\bullet(\bm{x};R)\cong K_\bullet(\bm{x}^{s.\flat};R^{s.\flat})
\]
in the derived category $D(\Ab)$ of abelian groups.
\end{theoremA}

Note that the theorem holds even for perfectoid \emph{rings} $A$ (cf.\ \cref{ex:perfectoid tower} (2)), which was already found in the proof of \cite[Proposition 16.4.18]{GR24}. However, the proof there relies on \emph{Fontaine's period ring} $\bbA_{\mathrm{inf}}(A)=W(A^\flat)$, which does not behave well for perfectoid towers.

On the other hand, we are interested in perfectoid towers consisting of Noetherian local rings. Recall that we have the following hierarchy of Noetherian local rings $(R,\m)$:
\begin{center}
regular $\Longrightarrow$ complete intersection $\Longrightarrow$ Gorenstein $\Longrightarrow$ Cohen--Macaulay.
\end{center}
These properties are characterized by several invariants such as the \emph{Krull dimension} $\dim(R)$, the \emph{embedding dimension} $\edim(R)=\dim_{R/\m}(\m/\m^2)$, the \emph{depth} $\depth(R)$, and the \emph{$q$-th deviation} $\varepsilon_q(R)=\dim_{R/\m}H_q(\bm{x})$, where $\bm{x}$ is a minimal basis of $\m$. In this context, the isomorphism \eqref{eq:tiltisom} is one of the most important methods for comparing the rings $R$ and $R^{s.\flat}$ (cf.\ \cite[\S 3D1]{INS25}). For example, if, moreover, $R$ is $I_0$-torsion free, then $R^{s.\flat}$ is $I_0^{s.\flat}$-torsion free, and thus \eqref{eq:tiltisom} implies that $R$ is Cohen--Macaulay (resp.\ Gorenstein, complete intersection) if and only if so is $R^{s.\flat}$. By applying \cref{thm:A}, we remove this torsion free assumption and obtain the following comparison result.

\begin{theoremA}[{\cref{prop:tilt-dim,thm:tilt-regular,thm:tilt-CM,thm:tilt-Gor}}]
Consider a perfectoid tower \eqref{eq:perfectoidtower} and its tilt \eqref{eq:tilt}. Assume that $R$ is a Noetherian local ring with perfect residue field. Then one has
\[
\dim(R)=\dim(R^{s.\flat}),\quad \edim(R)=\edim(R^{s.\flat}),\quad \depth(R)=\depth(R^{s.\flat}),\quad \varepsilon_q(R)=\varepsilon_q(R^{s.\flat})
\]
for all $q\in\Z$. Moreover, $R$ is Cohen--Macaulay of type $r$ \emph{(}resp.\ Gorenstein, complete intersection, hypersurface, regular\emph{)} if and only if so is $R^{s.\flat}$.
\end{theoremA}

Finally, we focus on \emph{$p$-Stanley--Reisner rings} as interest examples of rings having perfectoid towers with possible $p$-torsion elements; the notion of $p$-Stanley--Reisner rings was introduced by O.~Strahan \cite{Str25} as a mixed characteristic analogue of Stanley--Reisner rings.
It is known that every $p$-Stanley--Reisner ring with respect to an absolutely unramified complete discrete valuation ring (e.g., the Witt ring $W(k)$ of a perfect field $k$ of characteristic $p>0$) admits a perfectoid tower (\cref{ex:perfectoid tower} (4)). We apply the tilting invariance of Cohen--Macaulayness to deduce Reisner's criterion for such a $p$-Stanley--Reisner ring (\cref{cor:Reisner}).

%%%%%%%%%%%%%%%%%%%%%%%%%%%%%%%%%%%%%%%%%%%%%%%%%%%%%%%%%%%%%%%%%%%%%%%%%%%%%%%%
\medskip\noindent
\textbf{Acknowledgements.} % Acknowledgements section without numbering
The author would like to thank Kei Nakazato for continuous discussions and giving an idea of examples of non-reduced perfectoid towers, Ryo Ishizuka for discussing another proof of \cref{lem:Kos}. Moreover, the author is deeply thankful to Shinnosuke Ishiro, Nao Mochizuki, Kazuma Shimomoto, and Tatsuki Yamaguchi for their helpful comments.
Special thanks to Toshinori Kobayashi for calling his attention to the potential for further applications of tilting Koszul homology.
The author thanks Yoichi Mieda and Kanau Shimada for pointing out to an error in an earlier version of this paper.

\medskip\noindent
\textbf{Notations and conventions.}
  \begin{itemize}
  \item We consistently fix a prime number $p>0$.
  \item All rings are assumed to be commutative and unital (unless otherwise stated).
  \item For an $\F_p$-algebra $R$, let $\varphi=\varphi_R\colon R\to R$ denote the absolute Frobenius.
  \item By a \emph{pair} we simply mean a couple $(A,I)$ consisting of a ring $A$ and an ideal $I$ of $A$. When the ideal $I$ is principal, say $I=(a)$, then we often write $(A,a)$ in place of $(A,I)$.
  \item For a pair $(A,I)$ and an $A$-module $M$, we say that an element $x\in M$ is \emph{$I$-torsion} if for all $a\in I$ there exists an integer $n>0$ such that $a^nx=0$. Let $M_{\textrm{$I$-}\mathrm{tor}}$ denote the $A$-submodule of $M$ consisting of all $I$-torsion elements in $M$. We say that $M$ is \emph{$I$-torsion free} if $M_{\textrm{$I$-}\mathrm{tor}}=0$. Note that we follow the terminologies in \cite{FKI}. Let $\varphi_{I,M}\colon M_{\Itor}\to M/M_{\Itor}$ denote the composition of the inclusion $M_{\Itor}\hookrightarrow M$ followed by the canonical projection $M\twoheadrightarrow M/M_{\Itor}$.
  \item For a pair $(A,I)$, when we say an $A$-module $M$ is $I$-adically complete, we always mean that $M$ is Hausdorff complete with respect to the $I$-adic topology.
  \end{itemize}
%%%%%%%%%%%%%%%%%%%%%%%%%%%%%%%%%%%%%%%%%%%%%%%%%%%%%%%%%%%%%%%%%%%%%%%%%%%%%%%%
%%%%%%%%%%%%%%%%%%%%%%%%%%%%%%%%%%%%%%%%%%%%%%%%%%%%%%%%%%%%%%%%%%%%%%%%%%%%%%%%
\section{Preperfectoid towers and perfectoid towers}

In this section, we discuss fundamental properties of perfectoid towers.
We introduce a slightly more general notion, \emph{preperfectoid towers} (\cref{def:perfectoidtower}), and explain these differences (\cref{rem:pre}). After that, we review the tilts of (pre)perfectoid towers and prove a characterization of (pre)perfectoid towers in terms of cartesian diagrams (\cref{thm:PB}).

By a \emph{tower of rings} we simply mean an inductive system of rings $R_0\xr{t_0}R_1\xr{t_1}\cdots\to R_i\xr{t_i}\cdots$ indexed by $\N$, which we denote by $\textrm{{\boldmath $R$}}=\{R_i\}_{i\geq 0}=\{R_i,t_i\}_{i\geq 0}$. The \emph{category of towers of rings} is the category of inductive systems of rings. In particular, we can compute (co)limits of towers of rings levelwise.
Similarly, one also defines the category of towers of $R$-algebras for a fixed ring $R$ (e.g., $\F_p$).
In this paper, we use the following terminology.

\begin{definition}
\label{def:towerP}
Let $\textrm{{\boldmath $R$}}=\{R_i,t_i\}_{i\geq 0}$ be a tower of rings. For a property $P$ of rings or of $R_0$-modules, we say that \emph{{\boldmath $R$} satisfies $P$} if $R_i$ satisfies $P$ for every $i\geq 0$, where we regard $R_i$ as an $R_0$-module via the transition map $t_{i-1}\circ\cdots\circ t_0 \colon R_0\to R_i$.
\end{definition}

For example, we will consider the following properties: reduced, Noetherian, local, $I_0$-torsion free, $I_0$-adically separated, $I_0$-adically complete, where $I_0$ is an ideal of $R_0$.
Let us recall the following definitions.

\begin{definition}[{\cite[Definition 3.2]{INS25}}]
A tower of $\F_p$-algebras $\textrm{{\boldmath $R$}}=\{R_i,t_i\}_{i\geq 0}$ is called a \emph{perfect tower} if it is isomorphic to a tower of the form $R\xr{\varphi}R\xr{\varphi}R\xr{\varphi}\cdots$ for a reduced $\F_p$-algebra $R$.
\end{definition}

\begin{definition}[{\cite[Definition 3.4]{INS25}}]
\label{def:INS3.4}
Let $(R,I_0)$ be a pair. A \emph{purely inseparable tower arising from $(R,I_0)$} is a tower of rings $\textrm{{\boldmath $R$}}=\{R_i,t_i\}_{i\geq 0}$ satisfying the following conditions.
  \begin{itemize}
  \item[\textbf{(a)}] $R_0=R$ and $p\in I_0$.
  \item[\textbf{(b)}] For every $i\geq 0$, the ring homomorphism $\ol{t_i}\colon R_i/I_0R_i \to R_{i+1}/I_0R_{i+1}$ induced by $t_i$ is injective.
  \item[\textbf{(c)}] For every $i\geq 0$, we have $\Im(\varphi_{R_{i+1}/I_0R_{i+1}}) \subset \Im(\ol{t_i})$.
  \end{itemize}
In this case, for any $i\geq 0$ there exists a unique ring homomorphism $F_i\colon R_{i+1}/I_0R_{i+1}\to R_i/I_0R_i$ such that the diagram
\[
\begin{tikzcd}
R_{i+1}/I_0R_{i+1} \rar["\varphi"] \ar[rd,"F_i"'] & R_{i+1}/I_0R_{i+1} \\
 & R_i/I_0R_i \uar["\ol{t_i}"',hookrightarrow]
\end{tikzcd}
\]
commutes. We call $F_i$ the \emph{$i$-th Frobenius projection} (of {\boldmath $R$} associated to $(R,I_0)$).
\end{definition}

Perfectoid towers are defined as purely inseparable towers satisfying the additional conditions \textbf{(d)} $\sim$ \textbf{(g)}. Here condition \textbf{(e)} requires the property ``Zariskian.'' We define preperfectoid towers by removing this condition:

\begin{definition}[{cf.\ \cite[Definition 3.21]{INS25}}]
\label{def:perfectoidtower}
Let $(R,I_0)$ be a pair, and $\textrm{{\boldmath $R$}}=\{R_i,t_i\}_{i\geq 0}$ a tower of rings.
  \begin{enumerate}
  \item We say that {\boldmath $R$} is a \emph{preperfectoid tower arising from $(R,I_0)$} if it is a purely inseparable tower arising from $(R,I_0)$ satisfying the following additional conditions.
  \begin{itemize}
  \item[\textbf{(d)}] For every $i\geq 0$, the $i$-th Frobenius projection $F_i\colon R_{i+1}/I_0R_{i+1}\to R_i/I_0R_i$ is surjective.
  \item[\textbf{(f)}] $I_0$ is a principal ideal, and $R_1$ contains a principal ideal $I_1$ satisfying the following conditions.\footnote{The ideal $I_1$ is, if it exists, unique (\cite[Definition 3.25]{INS25}).}
    \begin{itemize}
    \item[\textbf{(f-1)}] $I_1^p=I_0R_1$.
    \item[\textbf{(f-2)}] For every $i\geq 0$, $\Ker(F_i)=I_1(R_{i+1}/I_0R_{i+1})$.
    \end{itemize}  
  We call $I_1$ the \emph{first perfectoid pillar} (of {\boldmath $R$} arising from $(R,I_0)$).
  \item[\textbf{(g)}] For every $i\geq 0$, $I_0(R_i)_{\textrm{$I_0$-}\mathrm{tor}}=(0)$. Moreover, there exists a bijection $(F_i)_{\mathrm{tor}}\colon (R_{i+1})_{\textrm{$I_0$-}\mathrm{tor}}\to (R_i)_{\textrm{$I_0$-}\mathrm{tor}}$ such that the diagram of sets
  \begin{equation*}
  %\label{eq:Ftor}
  \begin{tikzcd}[column sep=large]
  (R_{i+1})_{\textrm{$I_0$-}\mathrm{tor}} \rar["\varphi_{I_0,R_{i+1}}"] \dar["(F_i)_{\mathrm{tor}}"'] & R_{i+1}/I_0R_{i+1} \dar["F_i"] \\
  (R_i)_{\textrm{$I_0$-}\mathrm{tor}} \rar["\varphi_{I_0,R_i}"'] & R_i/I_0R_i
  \end{tikzcd}
  \end{equation*}
  commutes.\footnote{Since the map $\varphi_{I_0,R_i}$ is injective by \cite[Corollary 3.15]{INS25}, the map $(F_i)_{\mathrm{tor}}$ is, if it exists, unique and automatically a homomorphism of (possibly) non-unital rings.}%, where, for each $i\geq 0$, we denote by $\varphi_{I_0,R_i}$ the composition $(R_i)_{\Iztor}\hookrightarrow R_i\twoheadrightarrow R_i/I_0R_i$.
  \end{itemize}
  \item We say that {\boldmath $R$} is a \emph{perfectoid tower arising from $(R,I_0)$} if it is a preperfectoid tower arising from $(R,I_0)$ satisfying the following additional condition.
    \begin{itemize}
    \item[\textbf{(e)}] For every $i\geq 0$, $R_i$ is $I_0R_i$-adically Zariskian (i.e., $I_0R_i$ is contained in the Jacobson radical of $R_i$).
    \end{itemize}
  \end{enumerate}
\end{definition}

We first note differences between preperfectoid towers and perfectoid towers.

\begin{notation}
Let $(A,I)$ be a pair. Suppose that there exists a finite set of generators $x_1,\ldots,x_n$ for $I$ such that $p\in(x_1^p,\ldots,x_n^p)$. We then write
\[
I^{[p]}\coloneqq (x_1^p,\ldots,x_n^p).
\]
This ideal is independent of the choice of a finite set of generators for $I$ by \cite[Lemma 16.2.2]{GR24}.
\end{notation}

In \cite{INS25}, condition \textbf{(e)} was used only in \cite[Proposition 3.26]{INS25}, which can be refined as follows:

%\begin{lemma}
%Let $A$ be a ring, $I,J\subset A$ ideals. Then $I=J$ if and only if $I_\p=J_\p$ for every $\p\in\Spec A$.
%\end{lemma}
%
%\begin{proof}
%Since $I=J$ if and only if $I=I+J=J$, we may assume that $I\subset J$. Then $I=J$ if and only if $I\hookrightarrow J$ is surjective. Hence we obtain the assertion.
%\end{proof}

\begin{proposition}
\label{prop:Ii}
Let $\textrm{{\boldmath $R$}}=\{R_i,t_i\}_{i\geq 0}$ be a purely inseparable tower arising from a pair $(R,I_0)$, and assume that {\boldmath $R$} satisfies \emph{\textbf{(d)}} and \emph{\textbf{(f)}}. Let $\ol{R_i}\coloneqq R_i/I_0R_i$ for each $i\geq 0$.
  \begin{enumerate}
  \item For a sequence of ideals $\{\fraka_i\subset \ol{R_i}\}_{i\geq 0}$ such that $\Ker(F_i)\subset \fraka_{i+1}$ for every $i\geq 0$, the following conditions are equivalent.
  \begin{enumalph}
  \item $F_i^{-1}(\fraka_i)=\fraka_{i+1}$ for every $i\geq 0$.
  \item $F_i(\fraka_{i+1})=\fraka_i$ for every $i\geq 0$.
  \end{enumalph}
  \item Let $\{I_i\subset R_i\}_{i\geq 2}$ be a sequence of finitely generated ideals with $I_0R_i\subset I_i$ for every $i\geq 2$ such that $\{I_i\ol{R_i}\subset \ol{R_i}\}_{i\geq 0}$ satisfies the condition \emph{(b)} in \emph{(1)}.
    \begin{enumalph}
    \item $\Ker(F_i)\subset I_{i+1}\ol{R_{i+1}}$ for every $i\geq 0$. Hence $\{I_i\ol{R_i}\subset \ol{R_i}\}_{i\geq 0}$ satisfies the condition \emph{(a)} in \emph{(1)}.
    \item For any $i\geq 0$, there exists a system of generators $S$ for $I_{i+1}$ such that $p\in (S^p)$, and $I_{i+1}^{[p]}=I_iR_{i+1}$.
    \end{enumalph}
  \item There exists a unique sequence of finitely generated ideals $\{I_i\subset R_i\}_{i\geq 2}$ as in \emph{(2)}. Moreover, there exists a sequence of elements $\{f_i\in R_i\}_{i\geq 0}$ such that $I_i=(f_i)+I_0R_i$ and $F_i(\ol{f_{i+1}})=\ol{f_i}$ (where $\ol{f_i}$ denotes the image of $f_i$ in $\ol{R_i}$ for each $i\geq 0$).
  \end{enumerate}
\end{proposition}

\begin{proof}
(1) (a) $\Rightarrow$ (b): This follows from the surjectivity of $F_i$.

(b) $\Rightarrow$ (a): Fix $i\geq 0$. Then $F_i(\fraka_{i+1})=\fraka_i=F_i(F_i^{-1}(\fraka_i))$ implies $F_i^{-1}(\fraka_i)\subset F_i^{-1}(\fraka_i)+\Ker(F_i) = \fraka_{i+1}+\Ker(F_i) \subset F_i^{-1}(\fraka_i)$, and so $F_i^{-1}(\fraka_i)=\fraka_{i+1}+\Ker(F_i)$. But $\Ker(F_i)\subset \fraka_{i+1}$ by assumption, and thus $F_i^{-1}(\fraka_i)=\fraka_{i+1}$.
%This can be proved in exactly the same way as \cite[Proposition 3.26 (1)]{INS25}.

(2) (a) For every $i\geq 0$, the compatibility $\ol{t_i}\circ F_i=\varphi_{\ol{R_{i+1}}}$ implies $(I_{i+1}\ol{R_{i+1}})^{[p]}=I_i\ol{R_{i+1}}$. In particular, $\Ker(F_i)=I_1\ol{R_{i+1}}\subset I_{i+1}\ol{R_{i+1}}$.

(b) We proceed by induction on $i$.
If $i=0$, then what we want to show is $I_1^p=I_0R_1$, which is assumed by condition \textbf{(f-1)}.
We henceforth assume that the assertion is true with $i$ replaced by $j$ with $0\leq j\leq i-1$. 
Choose a generator of $I_0$, and extend it to a system of generators $\bm{x}=x_1,\ldots,x_n$ for $I_{i+1}$. Set $\bm{x}^p\coloneqq x_1^p,\ldots,x_n^p$. Then $(\bm{x}^p)\ol{R_{i+1}}=(I_{i+1}\ol{R_{i+1}})^{[p]} = I_i\ol{R_{i+1}}$ as we have seen above, and hence we obtain
\[
I_iR_{i+1} = (\bm{x}^p) + I_0R_{i+1} \subset (\bm{x}^p) + I_i^{p^i-1}(I_iR_{i+1}) \subset I_iR_{i+1}.
\]
Here we use the inclusion $I_0R_i\subset I_i^{p^i}$, which follows from the induction hypothesis. Thus we get the equalities of $R_{i+1}$-modules
\begin{equation}
\label{eq:Ii}
I_iR_{i+1} = (\bm{x}^p) + I_0R_{i+1} = (\bm{x}^p) + I_i^{p^i-1}(I_iR_{i+1}).
\end{equation}
Let us deduce from this that $(\bm{x}^p) = I_iR_{i+1}$. It suffices to show that $(\bm{x}^p)(R_{i+1})_\p = I_i(R_{i+1})_\p$ for every $\p\in\Spec R_{i+1}$. We proceed by considering the following cases.
  \begin{itemize}
  \item The case $I_iR_{i+1}\not\subset \p$. Since $I_iR_{i+1}=(\bm{x}^p)+I_0R_{i+1}\subset I_{i+1}$ by assumption, it follows that $I_{i+1}\not\subset\p$. Hence $(\bm{x})^p (R_{i+1})_\p = (R_{i+1})_\p = I_i (R_{i+1})_\p$.
  \item The case $I_iR_{i+1}\subset \p$. Then $I_i(R_{i+1})_\p\subset \p (R_{i+1})_\p$, and since $I_i$ is finitely generated, we can apply Nakayama's lemma to \eqref{eq:Ii} after localizing by $\p$.
  \end{itemize}

(3) If $\{I_i\subset R_i\}_{i\geq 2}$ is a sequence of (finitely generated) ideals as in (2), then $\Ker(F_i)\subset I_{i+1}\ol{R_{i+1}}$ for every $i\geq 0$ and the condition (a) in (1) is satisfied; therefore $I_{i+1}$ is the inverse image of $F_i^{-1}(I_i\ol{R_i})$ under the canonical projection $R_{i+1}\to\ol{R_{i+1}}$. Since $I_0$ is fixed, the uniqueness follows.

The existence follows from condition \textbf{(d)} by the axiom of (dependent) choice.
\end{proof}

\begin{remark}
\label{rem:pre}
In the situation as in \cref{prop:Ii}, $I_i$ is called the \emph{$i$-th perfectoid pillar of {\boldmath $R$} arising from $(R,I_0)$} (cf.\ \cite[Definition 3.27]{INS25}). Note that \cite[Proposition 3.26]{INS25} shows that if {\boldmath $R$} satisfies \textbf{(e)}, then we can take $I_i$ to be principal for every $i\geq 0$. 
This is the only difference between preperfectoid towers and perfectoid towers: to prove all results in \cite{INS25} (particularly, \S3C2, \S3D, \S3E), it is sufficient for the ideals $I_i\ol{R_i}$ of $\ol{R_i}$ to be principal.
One of merits of perfectoid towers is to obtain a new perfectoid tower by taking shifts: if $\textrm{{\boldmath $R$}}=\{R_i,t_i\}_{i\geq 0}$ is a perfectoid tower arising from $(R,I_0)$, then for every $j\geq 0$, $\{R_{j+i},t_{j+i}\}_{i\geq 0}$ is a perfectoid tower arising from $(R_j,I_j)$.
\end{remark}

Next let us review the definition of tilts of (pre)perfectoid towers.

\begin{definition}[{\cite[Definition 3.29, Definition 3.34]{INS25}}]
\label{def:tilt}
Let $\textrm{{\boldmath $R$}}=\{R_i,t_i\}_{i\geq 0}$ be a preperfectoid tower arising from a pair $(R,I_0)$. Then we define its \emph{tilt} $\textrm{{\boldmath $R$}}^\flat=\{R_i^{s.\flat},t_i^{s.\flat}\}_{i\geq 0}$ as follows.
  \begin{itemize}
  \item For any $i\geq 0$, let $R_i^{s.\flat}\coloneqq \varprojlim(\cdots\xr{F_{i+2}}R_{i+2}/I_0R_{i+2} \xr{F_{i+1}} R_{i+1}/I_0R_{i+1} \xr{F_i} R_i/I_0R_i)$. We call $R_i^{s.\flat}$ the \emph{$i$-th small tilt} (of {\boldmath $R$} associated to $(R,I_0)$).
  \item For any $i\geq 0$, let $t_i^{s.\flat}\colon R_i^{s.\flat}\to R_{i+1}^{s.\flat}$ be the ring homomorphism such that $\Phi^{(i+1)}_m \circ t_i^{s.\flat} = \ol{t_i}\circ \Phi^{(i)}_{m}$ for all $m\geq 0$, where $\Phi^{(i)}_{m} \colon R_i^{s.\flat}\to R_{i+m}/I_0R_{i+m}$ and $\Phi^{(i+1)}_m\colon R_{i+1}^{s.\flat}\to R_{i+m+1}/I_0R_{i+m+1}$ are the canonical projections.
  \end{itemize}
Moreover, let $I_0^{s.\flat}$ denote the kernel of the $0$-th projection $\Phi^{(0)}_0\colon R^{s.\flat}\to R/I_0R$.
In general, for each $i\geq 0$, the \emph{small tilt} $I_i^{s.\flat}$ of the $i$-th perfectoid pillar $I_i$ is the kernel of the $0$-th projection $\Phi^{(i)}_0\colon R_i^{s.\flat} \to R_i/I_0R_i$.
\end{definition}

One of the main theorems of \cite{INS25} is the following.

\begin{theorem}[{\cite[Theorem 3.35 (2), Lemma 3.39, Proposition 3.41]{INS25}}]
\label{thm:INS3.35}
Let $\textrm{{\boldmath $R$}}=\{R_i,t_i\}_{i\geq 0}$ be a preperfectoid tower arising from a pair $(R,I_0)$.
  \begin{enumerate}
  \item For every $i\geq 0$, we have an isomorphism of (possibly) non-unital rings $(R_i^{s.\flat})_{\Izstor}\xr{\cong}(R_i)_{\Iztor}$ that is compatible with $t_i^{s.\flat}$ and $t_i$. Moreover, the $0$-th projection $\Phi^{(i)}_0\colon R_i^{s.\flat}\to R_i/I_0R_i$ induces an isomorphism of rings $R_i^{s.\flat}/I_0^{s.\flat}R_i^{s.\flat}\xr{\cong} R_i/I_0R_i$ that is compatible with $t_i^{s.\flat}$ and $t_i$. Consequently, we have the commutative diagram of (possibly) non-unital rings
  \[
  \begin{tikzcd}[column sep=large]
  (R_i^{s.\flat})_{\Izstor} \rar["\varphi_{I_0^{s.\flat},R_i^{s.\flat}}",hookrightarrow] \dar["\cong"'] & R_i^{s.\flat}/I_0^{s.\flat}R_i^{s.\flat} \dar["\cong"] \\
  (R_i)_{\Iztor} \rar["\varphi_{I_0,R_i}",hookrightarrow] & R_i/I_0R_i.
  \end{tikzcd}
  \]
  \item The tilt $\textrm{{\boldmath $R$}}^\flat=\{R_i^{s.\flat},t_i^{s.\flat}\}_{i\geq 0}$ is a perfectoid tower arising from $(R^{s.\flat},I_0^{s.\flat})$.
  \end{enumerate}
\end{theorem}

Let us give some examples relevant to this paper.

\begin{example}
\label{ex:perfectoid tower}
  \begin{enumerate}
  \item For a reduced $\F_p$-algebra $R$, the tower consisting of absolute Frobenius
  \begin{equation}
  \label{eq:perfect tower}
  R\xr{\varphi}R\xr{\varphi}R\xr{\varphi}\cdots
  \end{equation}
  is a perfectoid tower arising from $(R,0)$. A tower that is isomorphic to \eqref{eq:perfect tower} for a reduced $\F_p$-algebra is called a \emph{perfect tower} (\cite[Definition 3.2]{INS25}). A tower of rings is a perfect tower if and only if it is a perfectoid tower arising from $(R,0)$ (\cite[Lemma 3.24]{INS25}).  %Note that \eqref{eq:perfect tower} is isomorphic to the tower consisting of the ring of $p$-power roots $R\hookrightarrow R^{1/p}\hookrightarrow R^{1/p^2}\hookrightarrow\cdots$ (\cite[Example 3.3]{INS25}).
  \item Let $A$ be a perfectoid ring that is $\varpi$-adically complete for a $\varpi\in A$ with $p\in \varpi^pA$. Then
  \[
  A\xr{\id}A\xr{\id}A\xr{\id}\cdots
  \]
  is a perfectoid tower arising from $(A,\varpi^p)$ (\cite[Example 3.53]{INS25}). Its tilt is
  \[
  A^\flat \xr{\id} A^\flat \xr{\id} A^\flat \xr{\id} \cdots,
  \]
  where $A^\flat=\varprojlim_\varphi(A/\varpi A)$ is the \emph{tilt} of $A$.  
  \item Fix an algebraic closure $\ol{\Q_p}$ of $\Q_p$, and choose a compatible system $\{p^{1/p^i}\}_{i\geq 0}$ of $p$-power roots of $p$ in $\ol{\Q_p}$. Then
  \[
  \Z_p\hookrightarrow \Z_p[p^{1/p}]\hookrightarrow \Z_p[p^{1/p^2}]\hookrightarrow\cdots
  \]
  is a perfectoid tower arising from $(\Z_p,p)$ (\cite[Example 3.23 (1)]{INS25}). Its tilt is isomorphic to the tower
  \[
  \F_p\llbracket T\rrbracket\hookrightarrow \F_p\llbracket T^{1/p}\rrbracket\hookrightarrow \F_p\llbracket T^{1/p^2}\rrbracket \hookrightarrow\cdots.
  \]
  A similar construction works for complete local log-regular rings (\cite[\S3F]{INS25}).
  \item %\emph{Quotients by squarefree monomial ideals} (\cite[\S6.4]{Ish26}, \cite[\S4]{IS25}). 
  In his paper \cite{Ish26}, Ishizuka constructed perfectoid towers arising from \emph{Frobenius lifts}. After that, Ishiro--Shimomoto \cite{IS25} gave another proof, which is based on a direct calculation but includes the $p$-torsion case. Here we introduce, as an interest example, perfectoid towers arising from \emph{$p$-Stanley--Reisner rings}. Let $\Delta$ be a simplicial complex on the vertex set $\{v_1,\ldots,v_n\}$, and $C$ a ring. The \emph{$p$-Stanley--Reisner ring of $\Delta$ with respect to $C$} is the homogeneous $C$-algebra
  \[
  C\ol{[\Delta]} = C[x_2,\ldots,x_n]/\ol{I_\Delta},
  \]
  where $\ol{I_\Delta}$ is the ideal generated by all $p$-monomials $x_{i_1}x_{i_2}\cdots x_{i_s}$ ($x_1\coloneqq p$) such that $\{v_{i_1},\ldots,v_{i_s}\}\notin\Delta$.
  Now we consider the Cohen ring $C(k)$ of a field $k$ of characteristic $p>0$.\footnote{For a field $k$ of characteristic $p>0$, the \emph{Cohen ring} $C(k)$ (according to the terminology in \cite[Definition 1.53]{FO}) is the unique (up to isomorphism) absolutely unramified complete discrete valuation ring of mixed characteristic $(0,p)$ whose residue field is $k$.}
  Then a conclusion of constructions due to Ishizuka, Ishiro and Shimomoto is that $C(k)\ol{[\Delta]}$ admits a preperfectoid tower and its $0$-th small tilt is given by the Stanley--Reisner ring $k[\Delta]$ (\cite[Example 6.12]{Ish26}, \cite[Theorem 4.6, Proposition 4.26]{IS25}).
  \end{enumerate}
\end{example}

Using \cref{thm:INS3.35}, we establish our first main theorem that characterizes condition \textbf{(g)} in terms of cartesian diagrams. This result will be used in the next section. Let us start with the following lemma.

\begin{lemma}
\label{lem:tiltinj}
Let $\textrm{{\boldmath $R$}}=\{R_i,t_i\}_{i\geq 0}$ be a preperfectoid tower arising from a pair $(R,I_0)$. Then for any $i\geq 0$, the ring homomorphism $R^{s.\flat}/(I_0^{s.\flat})^{p^i}\to R^{s.\flat}/(I_0^{s.\flat})^{p^{i+1}}$ induced from the absolute Frobenius of $R^{s.\flat}$ is injective.
\end{lemma}

\begin{proof}
The tilt $\textrm{{\boldmath $R$}}^\flat=\{R_i^{s.\flat},t_i^{s.\flat}\}_{i\geq 0}$ is a perfect tower (\cite[Proposition 3.10 (2)]{INS25}). But $\textrm{{\boldmath $R$}}^\flat$ is also a perfectoid tower arising from $(R^{s.\flat},I_0^{s.\flat})$ by \cref{thm:INS3.35} (2), and thus condition \textbf{(b)} in \cref{def:INS3.4} implies the assertion.
\end{proof}

%\begin{lemma}
%\label{lem:PB}
%In the category of sets, consider the commutative diagram of solid arrows
%\[
%\begin{tikzcd}
%X \rar[dashed,"x"] \dar["f"'] \ar[rr,bend left,"x_0"] & X' \rar["x'"] \dar["f'"] & X'' \dar["f''"] \\
%Y \rar["y"'] & Y' \rar["y'"'] & Y''
%\end{tikzcd}
%\]
%Assume that the right-hand square is cartesian.
%  \begin{enumerate}
%  \item There uniquely exists a dotted arrow $x$ such that $x'\circ x=x_0$ and $f'\circ x=y\circ f$.
%  \item Assume that the outer square is cartesian. If $y$ is bijective, both $f$ and $y'$ are injective, then $x$ is bijective.
%  \end{enumerate}
%\end{lemma}
%
%\begin{proof}
%(1) immediately follows from the universal property of fiber products.
%
%(2) Since $f'\circ x=y\circ f$ is injective, so is $x$. 
%\end{proof}

\begin{theorem}
\label{thm:PB}
Let $\textrm{{\boldmath $R$}}=\{R_i,t_i\}_{i\geq 0}$ be a purely inseparable tower arising from a pair $(R,I_0)$. Then, under \emph{\textbf{(d)}} and \emph{\textbf{(f)}}, the condition \emph{\textbf{(g)}} is equivalent to the following one:
  \begin{itemize}
  \item For every $i\geq 0$, the diagrams of sets
  \begin{equation}
  \label{eq:PBt}
  \begin{tikzcd}[column sep=large]
  (R_i)_{\Iztor} \rar["\varphi_{I_0,R_i}"] \dar["(t_i)_{\mathrm{tor}}"',hookrightarrow] & R_i/I_0R_i \dar["\ol{t_i}",hookrightarrow] \\
  (R_{i+1})_{\Iztor} \rar["\varphi_{I_0,R_{i+1}}"] & R_{i+1}/I_0R_{i+1}
  \end{tikzcd}
  \end{equation}
  and
  \begin{equation}
  \label{eq:PBphi}
  \begin{tikzcd}[column sep=large]
  (R_{i+1})_{\Iztor} \rar["\varphi_{I_1,R_{i+1}}"] \dar["\varphi"'] & R_{i+1}/I_1R_{i+1} \dar["\varphi'",hookrightarrow] \\
  (R_{i+1})_{\Iztor} \rar["\varphi_{I_0,R_{i+1}}"] & R_{i+1}/I_0R_{i+1}
  \end{tikzcd}
  \end{equation}
  are cartesian, where $(t_i)_{\mathrm{tor}}$ is the restriction of $t_i$, $\varphi'$ is the injection induced by the absolute Frobenius of $R_{i+1}/I_0R_{i+1}$ (cf.\ \cite[Lemma 3.6 (1)]{INS25}), and $\varphi$ is given by $x\mapsto x^p$.
  \end{itemize}
\end{theorem}

\begin{proof}
\textbf{(The ``only if'' part)}: Fix $i\geq 0$. First we show that \eqref{eq:PBt} is cartesian if and only if so is \eqref{eq:PBphi}. We have the commutative diagram of sets
\begin{equation}
\label{eq:PBtphi}
\begin{tikzcd}
(R_{i+1})_{\Iztor} \rar["(F_i)_{\mathrm{tor}}","\cong"'] \dar["\varphi_{I_1,R_{i+1}}"'] & (R_i)_{\Iztor} \rar["(t_i)_{\mathrm{tor}}",hookrightarrow] \dar["\varphi_{I_0,R_i}"] & (R_{i+1})_{\Iztor} \dar["\varphi_{I_0,R_{i+1}}"] \\
R_{i+1}/I_1R_{i+1} \rar["F'_i"',"\cong"] & R_i/I_0R_i \rar["\ol{t_i}"',hookrightarrow] & R_{i+1}/I_0R_{i+1},
\end{tikzcd}
\end{equation}
where $F'_i$ is the isomorphism induced from the Frobenius projection $F_i\colon R_{i+1}/I_0R_{i+1}\to R_i/I_0R_i$. Thus the right-hand square is cartesian if and only if so is the outer square. Since the right-hand square is nothing but \eqref{eq:PBt} and the outer square is nothing but \eqref{eq:PBphi}, our claim follows.

Next we show that \eqref{eq:PBt} is cartesian. By \cref{thm:INS3.35} (1), we may replace $\textrm{{\boldmath $R$}}=\{R_j,t_j\}_{j\geq 0}$ by $\textrm{{\boldmath $R$}}^\flat=\{R_j^{s.\flat},t_j^{s.\flat}\}_{j\geq 0}$. Then by \cref{lem:tiltinj} and the fact that $\{R_{i+j},t_{i+j}\}_{j\geq 0}$ is a perfectoid tower arising from $(R_i,I_0R_i)$, we are reduced to show that the following diagram of (possibly) non-unital rings is cartesian.
\[
\begin{tikzcd}
R_{\Iztor} \rar["\varphi_{I_0,R}"] \dar["\varphi"',hookrightarrow] & R/I_0 \dar[hookrightarrow] \\
R_{\Iztor} \rar["\varphi_{I_0^p,R}"] & R/I_0^p.
\end{tikzcd}
\]
where the right hand vertical arrow is induced from the absolute Frobenius of $R$.
Choose a generator $f_0$ of $I_0$, and take $x\ \mathrm{mod}\ I_0\in R/I_0$ such that $x^p\ \mathrm{mod}\ I_0^p \in \Im(\varphi_{I_0^p,R})$. Then
\[
(f_0^{p-1}x)^p = f_0^{p(p-1)}x^p \in I_0^{p(p-1)+p} = I_0^{p^2}.
\]
Since the $p$-th power map $R/I_0^p\to R/I_0^{p^2}$ is injective (\cref{lem:tiltinj}), we have $f_0^{p-1}x \in I_0^p$. Hence we can write $f_0^{p-1}x = f_0^py$ for some $y\in R$. Then $f_0^{p-1}(x-f_0y)=0$, and so $x\ \mathrm{mod}\ I_0 = \varphi_{I_0,R}(x-f_0y)$.

\textbf{(The ``if'' part)}: Fix $i\geq 0$. By assumption, both $(t_i)_{\mathrm{tor}}$ and $\varphi\colon(R_{i+1})_{\Iztor}\to(R_{i+1})_{\Iztor}$ are injective, and thus $I_0(R_i)_{\Iztor}\hookrightarrow I_0(R_{i+1})_{\Iztor}=(0)$ by \cite[Lemma 3.18]{INS25}. The rest part is a formal consequence: since \eqref{eq:PBt} is cartesian, there exists a unique map $(F_i)_{\mathrm{tor}}$ as in \eqref{eq:PBtphi}. Then, in the commutative diagram \eqref{eq:PBtphi}, the outer square and the right-hand square are cartesian. Hence the left-hand square is also cartesian. Since $F'_i$ is an isomorphism, so is $(F_i)_{\mathrm{tor}}$.
\end{proof}

Finally, we define tilts of morphisms of purely inseparable towers.

\begin{definition}
Let $\textrm{{\boldmath $R$}}=\{R_i,t_i\}_{i\geq 0}$ and $\textrm{{\boldmath $S$}}=\{S_i,u_i\}_{i\geq 0}$ be purely inseparable towers arising from $(R,I_0)$ and $(S,J_0)$, respectively.
Let $\textrm{{\boldmath $\alpha$}}=\{\alpha_i\}_{i\geq 0}\colon \textrm{{\boldmath $R$}} \to \textrm{{\boldmath $S$}}$ be a morphism of towers of rings such that $\alpha_0(I_0)\subset J_0$. Then for any $i\geq 0$ the right-hand and outer squares in the diagram
\[
\begin{tikzcd}
R_{i+1}/I_0R_{i+1} \rar["F_i"] \dar["\ol{\alpha_{i+1}}"'] & R_i/I_0R_i \rar["\ol{t_i}"] \dar["\ol{\alpha_i}"] & R_{i+1}/I_0R_{i+1} \dar["\ol{\alpha_{i+1}}"] \\
S_{i+1}/J_0S_{i+1} \rar["G_i"] & S_i/J_0S_i \rar["\ol{u_i}"] & S_{i+1}/J_0S_{i+1}
\end{tikzcd}
\]
commute, where $F_i$ and $G_i$ are the $i$-th Frobenius projections for {\boldmath $R$} and {\boldmath $S$}, respectively. But $\ol{u_i}$ is injective, and so the left-hand square also commutes. Hence {\boldmath $\alpha$} induces a morphism of towers of rings $\textrm{{\boldmath $\alpha$}}^{\mathrm{frep}}=\{\alpha_i^{\mathrm{q.frep}}\}_{i\geq 0}\colon \textrm{{\boldmath $R$}}^{\mathrm{q.frep}} \to \textrm{{\boldmath $S$}}^{\mathrm{q.frep}}$. We call {\boldmath $\alpha$} the \emph{inverse perfection of {\boldmath $\alpha$} (associated to $(R,I_0)$ and $(S,J_0)$)}. If {\boldmath $R$} and {\boldmath $S$} are preperfectoid towers, then we define the \emph{tilt of {\boldmath $\alpha$} (associated to $(R,I_0)$ and $(S,J_0)$)} as $\textrm{{\boldmath $\alpha$}}^{\mathrm{q.frep}}$ and denote it by $\textrm{{\boldmath $\alpha$}}^\flat=\{\alpha_i^{s.\flat}\}_{i\geq 0}\colon \textrm{{\boldmath $R$}}^\flat \to \textrm{{\boldmath $S$}}^\flat$.
\end{definition}

The following lemma is useful when we consider the small tilt $J_0^{s.\flat}\subset S^{s.\flat}$ of $J_0$.

\begin{lemma}
\label{lem:alphaflat}
Let $\textrm{{\boldmath $R$}}=\{R_i,t_i\}_{i\geq 0}$ and $\textrm{{\boldmath $S$}}=\{S_i,u_i\}_{i\geq 0}$ be preperfectoid towers arising from $(R,I_0)$ and $(S,J_0)$, respectively. Let $\textrm{{\boldmath $\alpha$}}=\{\alpha_i\}_{i\geq 0}\colon \textrm{{\boldmath $R$}} \to \textrm{{\boldmath $S$}}$ be a morphism of towers of rings such that $I_0S=J_0$ and $I_1S_1=J_1$. Fix $i\geq 0$.
  \begin{enumerate}
  \item The $i$-th perfectoid pillar $J_i$ of {\boldmath $S$} is $I_iS_i$, where $I_i$ is the $i$-th perfectoid pillar of {\boldmath $R$}.
  \item The small tilt $J_i^{s.\flat}\subset S_i^{s.\flat}$ of $J_i$ is $I_i^{s.\flat}S_i^{s.\flat}$, where $I_i^{s.\flat}\subset R_i^{s.\flat}$ is the small tilt of $I_i$.
  \end{enumerate}
\end{lemma}

\begin{proof}
For each let $\ol{R_i}\coloneqq R_i/I_0R_i$ and $\ol{S_i}\coloneqq S_i/J_0S_i=S_i/I_0S_i$. Moreover, let $F_i\colon \ol{R_{i+1}}\to\ol{R_i}$ and $G_i\colon \ol{S_{i+1}}\to \ol{S_i}$ denote the $i$-th Frobenius projections.

(1) Since $\Ker(G_i) = J_1\ol{S_{i+1}} = I_1\ol{S_{i+1}} \subset I_{i+1}\ol{S_{i+1}}$ for every $i\geq 0$, it suffices to show that $G_i(I_{i+1}\ol{S_{i+1}})=I_i\ol{S_i}$ for every $i\geq 0$ (\cref{prop:Ii} (1)). By the commutative diagram
\[
\begin{tikzcd}
\ol{R_{i+1}} \rar["F_i"] \dar["\ol{\alpha_{i+1}}"'] & \ol{R_i} \dar["\ol{\alpha_i}"] \\
\ol{S_{i+1}} \rar["G_i"] & \ol{S_i},
\end{tikzcd}
\]
we immediately have the assertion as follows:
\[
G_i(I_{i+1}\ol{S_{i+1}}) = F_i(I_{i+1}\ol{R_{i+1}})\cdot \ol{S_i} = I_i\ol{R_i}\cdot\ol{S_i} = I_i\ol{S_i}.
\]

(2) Choose a generator $f_i^{s.\flat}$ of $I_i^{s.\flat}$. For each $j\geq 0$, let $\ol{f_{i+j}}\in \ol{R_{i+j}}$ denote the image of $f_i^{s.\flat}$. Then $\ol{f_{i+j}}$ is a generator of $I_{i+j}\ol{R_{i+j}}$ by \cite[Theorem 3.35 (1)]{INS25}, and thus $\ol{\alpha_{i+j}}(\ol{f_{i+j}})$ is a generator of $I_{i+j}\ol{S_{i+j}}$. But $I_{i+j}\ol{S_{i+j}}$ is the ($i+j$)-th perfectoid pillar. Hence, again by \cite[Theorem 3.35 (1)]{INS25}, $\alpha_i^{s.\flat}(f_i^{s.\flat})$ is a generator of $J_i^{s.\flat}$. This means that $J_i^{s.\flat}=I_i^{s.\flat}S_i^{s.\flat}$.
\end{proof}

%%%%%%%%%%%%%%%%%%%%%%%%%%%%%%%%%%%%%%%%%%%%%%%%%%%%%%%%%%%%%%%%%%%%%%%%%%%%%%%%%%%%%%%%%%%%%%%
%%%%%%%%%%%%%%%%%%%%%%%%%%%%%%%%%%%%%%%%%%%%%%%%%%%%%%%%%%%%%%%%%%%%%%%%%%%%%%%%%%%%%%%%%%%%%%%
\section{Structural properties of perfectoid towers}
\label{s:str}

In this section, we prove that perfectoid towers admit a canonical decomposition into the fiber product of $p$-torsion free perfectoid towers and perfect towers (\cref{thm:decomp}), and conversely that they can be obtained from such fiber products (\cref{prop:gluetower}). We also discuss several consequences of these theorems. As stated in the introduction, these results are tower-theoretic analogues of the corresponding results for perfectoid rings.

%%%%%%%%%%%%%%%%%%%%%%%%%%%%%%%%%%%%%%%%%%%%%%%%%%%%%%%%%%%%%%%%%%%%%%%%%%%%%%%%
\subsection{Canonical decompositions of perfectoid towers}

%We will show that the construction as in \cref{prop:gluetower} is essentially all we can do;
In this subsection we are going to prove the following theorem.

\begin{theorem}
\label{thm:decomp}
Let $\textrm{{\boldmath $R$}} = \{R_i,t_i\}_{i\geq 0}$ be a preperfectoid (resp.\ perfectoid) tower arising from a pair $(R,I_0)$, and consider the induced towers of rings
\begin{align*}
\wt{\textrm{{\boldmath $R$}}}&\colon\quad R_0/(R_0)_{\Iztor} \to R_1/(R_1)_{\Iztor} \to \cdots \to \wt{R_i}\coloneqq R_i/(R_i)_{\Iztor} \to \cdots, \\
(\textrm{{\boldmath $R$}}/I_0\textrm{{\boldmath $R$}})_{\mathrm{red}} &\colon\quad (R_0/I_0)_\red \to (R_1/I_0R_1)_\red \to \cdots \to (R_i/I_0R_i)_\red \to \cdots, \\
(\wt{\textrm{{\boldmath $R$}}}/I_0\wt{\textrm{{\boldmath $R$}}})_{\mathrm{red}} &\colon\quad 
(\wt{R_0}/I_0\wt{R_0})_\red \to (\wt{R_1}/I_0\wt{R_1})_\red \to \cdots \to (\wt{R_i}/I_0\wt{R_i})_\red \to \cdots.
\end{align*}
  \begin{enumerate}
  \item $\wt{\textrm{{\boldmath $R$}}}$ is a preperfectoid (resp.\ perfectoid) tower arising from $(S,I_0S)$.
  \item $\left(\textrm{{\boldmath $R$}}/I_0\textrm{{\boldmath $R$}}\right)_{\mathrm{red}}$ and $\left(\wt{\textrm{{\boldmath $R$}}}/I_0\wt{\textrm{{\boldmath $R$}}}\right)_{\mathrm{red}}$ are perfect towers.
  \item The commutative diagrams
  \[
  \begin{tikzcd}
  \textrm{{\boldmath $R$}} \rar \dar & \wt{\textrm{{\boldmath $R$}}} \dar \\
  (\textrm{{\boldmath $R$}}/I_0\textrm{{\boldmath $R$}})_{\mathrm{red}} \rar & (\wt{\textrm{{\boldmath $R$}}}/I_0\wt{\textrm{{\boldmath $R$}}})_{\mathrm{red}}
  \end{tikzcd}
  \qquad
  \begin{tikzcd}
  \textrm{{\boldmath $R$}}^\flat \rar \dar & (\wt{\textrm{{\boldmath $R$}}})^\flat \dar \\
  (\textrm{{\boldmath $R$}}/I_0\textrm{{\boldmath $R$}})_{\mathrm{red}} \rar & (\wt{\textrm{{\boldmath $R$}}}/I_0\wt{\textrm{{\boldmath $R$}}})_{\mathrm{red}}
  \end{tikzcd}
  \]
  consisting of canonical projections of towers of rings are cartesian.
  \end{enumerate}
\end{theorem}

We first study the tower $\wt{\textrm{{\boldmath $R$}}}$ consisting of the maximal $I_0$-torsion free quotients.

\begin{proposition}
\label{prop:torftower}
Let $\textrm{{\boldmath $R$}}=\{R_i,t_i\}_{i\geq 0}$ be a preperfectoid tower arising from a pair $(R,I_0)$, and consider the induced tower of rings
\[
\wt{\textrm{{\boldmath $R$}}} \colon \quad \wt{R}\coloneqq R/R_{\Iztor} \xr{\wt{t_0}} \wt{R_1}\coloneqq R_1/(R_1)_{\Iztor} \xr{\wt{t_1}} \cdots \to \wt{R_i}\coloneqq R_i/(R_i)_{\Iztor} \xr{\wt{t_i}} \cdots.
\]
  \begin{enumerate}
  \item $\wt{\textrm{{\boldmath $R$}}}$ is a preperfectoid tower arising from $(\wt{R},I_0\wt{R})$.
  \item If {\boldmath $R$} is a perfectoid tower arising from $(R,I_0)$, then $\wt{\textrm{{\boldmath $R$}}}$ is a perfectoid tower arising from $(\wt{R},I_0\wt{R})$.
  \item The tilt $(\wt{\textrm{{\boldmath $R$}}})^\flat$ of $\wt{\textrm{{\boldmath $R$}}}$ is given by
  \[
  R^{s.\flat}/(R^{s.\flat})_{\Izstor} = R_0^{s.\flat}/(R_0^{s.\flat})_{\Izstor} \to R_1^{s.\flat}/(R_1^{s.\flat})_{\Izstor} \to \cdots \to R_i^{s.\flat}/(R_i^{s.\flat})_{\Izstor} \to \cdots.
  \]
  \item The $i$-th perfectoid pillar of $\wt{\textrm{{\boldmath $R$}}}$ is $I_i\wt{R_i}$, where $I_i$ is the $i$-th perfectoid pillar of {\boldmath $R$}. Moreover, the small tilt of $I_i\wt{R_i}$ is $I_i^{s.\flat}\wt{R_i}^{s.\flat}$, where $I_i^{s.\flat}\subset R_i^{s.\flat}$ is the small tilt of $I_i$.
  \item The $I_1\wt{R_1}$-completed colimit $(\wt{R}_\infty)^\wedge$ of the perfectoid tower $\wt{\textrm{{\boldmath $R$}}}$ is isomorphic to the maximal $I_0$-torsion free quotient $\wh{R_\infty}/(\wh{R_\infty})_{\Iztor}$ of the perfectoid ring $\wh{R_\infty}$ (cf.\ \cite[Corollary 3.52]{INS25}).
  \end{enumerate}
\end{proposition}

\begin{proof}
We first note that for every $i\geq 0$, since the map $\varphi_{I_0,R_i}\colon (R_i)_{\Iztor}\to R_i/I_0R_i$ is injective, it yields the exact sequence of $R_i$-modules
\begin{equation}
\label{eq:torfquot}
0 \to (R_i)_{\textrm{$I_0$-}\mathrm{tor}} \xr{\varphi_{I_0,R_i}} R_i/I_0R_i \to \wt{R_i}/I_0\wt{R_i} \to 0.
\end{equation}

(1) We are going to check the conditions \textbf{(a)} $\sim$ \textbf{(d)}, \textbf{(f)}, \textbf{(g)} in \cref{def:perfectoidtower}.

\textbf{(a)} is obvious.

\textbf{(b)} By \eqref{eq:torfquot}, we have the commutative diagram of abelian groups with exact rows
\[
\begin{tikzcd}
0\rar & (R_i)_{\Iztor} \rar[hookrightarrow] \dar["(t_i)_{\mathrm{tor}}"',hookrightarrow] \ar[rd,phantom,"\textrm{p.b.}"] & R_i/I_0R_i \dar["\ol{t_i}",hookrightarrow] \rar & \wt{R_i}/I_0\wt{R_i} \rar \dar["\ol{t_i}"] & 0 \\
0\rar & (R_{i+1})_{\Iztor} \rar[hookrightarrow] & R_{i+1}/I_0R_{i+1} \rar & \wt{R_{i+1}}/I_0\wt{R_{i+1}}\rar & 0,
\end{tikzcd}
\]
where the left-hand square is cartesian by \cref{thm:PB}. We deduce from this that the map $\ol{u_i}$ is injective.

\textbf{(c)}, \textbf{(d)} and \textbf{(f)}: By \eqref{eq:torfquot}, we have the commutative diagram of abelian groups with exact rows and columns
\begin{equation}
\label{eq:torfF}
\begin{tikzcd}
 & & 0 \dar & 0 \dar & \\
 & & I_1(R_{i+1}/I_0R_{i+1}) \rar[dashed,"\cong"] \dar & I_1(\wt{R_{i+1}}/I_0\wt{R_{i+1}}) \dar & \\
0 \rar & (R_{i+1})_{\textrm{$I_0$-}\mathrm{tor}} \rar \dar["\cong","(F_i)_{\mathrm{tor}}"'] & R_{i+1}/I_0R_{i+1} \rar \dar["F_i"] & \wt{R_{i+1}}/I_0\wt{R_{i+1}} \rar \dar["\wt{F_i}"] & 0 \\
0 \rar & (R_i)_{\textrm{$I_0$-}\mathrm{tor}} \rar & R_i/I_0R_i \rar \dar & \wt{R_i}/I_0\wt{R_i} \rar \dar & 0 \\
 & & 0 & 0, & 
\end{tikzcd}
\end{equation}
where $\wt{F_i}$ is the map induced by the universality of cokernels, and by an easy diagram chasing we deduce that the dotted arrow is an isomorphism. Since $F_i$ is a ring homomorphism, so is $\wt{F_i}$. Thus we see that $\wt{F_i}$ gives the $i$-th Frobenius projection for $\wt{\textrm{{\boldmath $R$}}}$, and that the tower $\wt{\textrm{{\boldmath $S$}}}$ satisfies \textbf{(d)} and \textbf{(f)}.

\textbf{(g)} follows automatically, since $\wt{R_i}=R_i/(R_i)_{\Iztor}$ is $I_0$-torsion free for every $i\geq 0$.

(2) follows because we have the canonical surjection $R_i\to \wt{R_i}$ for every $i\geq 0$.

(3) Fix $i\geq 0$. By taking projective limits of \eqref{eq:torfF}, we have the exact sequence of $R_i^{s.\flat}$-modules
\[
0\to \varprojlim_{j\geq 0}(R_{i+j})_{\Iztor} \to R_i^{s.\flat} \to (\wt{R_i})^{s.\flat} \to 0.
\]
But $\varprojlim_{j\geq 0}(R_{i+j})_{\Iztor} \cong (R_i^{s.\flat})_{\Izstor}$ by the proof of \cite[Theorem 3.35 (2)]{INS25}, and so the assertion follows.

(4) By the proof of (1), the first perfectoid pillar of $\wt{\textrm{{\boldmath $R$}}}$ is $I_1\wt{R_1}$. Hence the assertion follows from \cref{lem:alphaflat}.

(5) By taking inductive limits of the canonical exact sequences $0 \to (R_i)_{\Iztor} \to R_i \to \wt{R_i}\to 0$, we have the exact sequence of $R_\infty$-modules
\[
0\to (R_\infty)_{\Iztor} \to R_\infty \to \wt{R}_\infty \to 0.
\]
Since both $R_\infty$ and $\wt{R}_\infty$ have bounded $I_0$-torsion, the sequence
\[
0\to (R_\infty)_{\Iztor} \to \wh{R_\infty} \to (\wt{R}_\infty)^\wedge \to 0
\]
consisting of $I_0$-adically completions is exact (\cite[Chapter 0, Lemma 8.2.14]{FKI}). Moreover, by \cite[Lemma 3.16]{INS25}, the canonical map $R_\infty\to\wh{R_\infty}$ induces an isomorphism $(R_\infty)_{\Iztor}\xr{\cong}(\wh{R_\infty})_{\Iztor}$. This confirms the assertion.
\end{proof}

We turn to explain how to obtain a perfect tower from a perfecotid tower. The following lemma follows immediately from the definition or \cite[Lemma 3.24]{INS25}.

\begin{lemma}
\label{lem:perfecttower}
A tower of $\F_p$-algebras $\textrm{{\boldmath $R$}}=\{R_i,t_i\}_{i\geq 0}$ is a perfect tower if and only if it satisfies the following conditions.
  \begin{enumerate}
  \item $R_0$ is reduced.
  \item For any $i\geq 0$, the map $t_i$ is injective.
  \item For any $i\geq 0$, we have $R_{i+1}^p=\Im(t_i)$, where $R_{i+1}^p\coloneqq\{x^p\mid x\in R_{i+1}\}$. In other words, there exists a (unique) ring isomorphism $F_i\colon R_{i+1}\to R_i$ such that the following diagram commutes.
  \[
  \begin{tikzcd}
  R_{i+1} \rar["\varphi"] \ar[rd,"F_i"',"\cong" sloped] & R_{i+1} \\
   & R_i \uar["t_i"',hookrightarrow]
  \end{tikzcd}
  \]
  \end{enumerate}
\end{lemma}

\begin{proposition}
\label{prop:modIred}
Let $\textrm{{\boldmath $R$}}=\{R_i,t_i\}_{i\geq 0}$ be a purely inseparable tower arising from a pair $(R,I_0)$, and assume that it satisfies \emph{\textbf{(d)}}. Then the induced tower of $\F_p$-algebras
\[
(\textrm{{\boldmath $R$}}/I_0\textrm{{\boldmath $R$}})_\red \colon\quad (R/I_0)_\red = (R_0/I_0)_\red \xr{(\ol{t_0})_\red} (R_1/I_0R_1)_\red \xr{(\ol{t_1})_\red} \cdots \to (R_i/I_0R_i)_\red \xr{(\ol{t_i})_\red} \cdots
\]
is a perfect tower.
\end{proposition}

\begin{proof}
We are going to check (1), (2), and (3) in \cref{lem:perfecttower}.

(1) is obvious.

(2) Since $\ol{t_i}$ is injective, so is $(\ol{t_i})_\red$.

(3) By applying $(-)_\red$ to the compatibility $\ol{t_i}\circ F_i = \varphi_{R_{i+1}/I_0R_{i+1}}$, we have the commutative diagram of rings
\[
\begin{tikzcd}
(R_{i+1}/I_0R_{i+1})_\red \rar["\varphi"] \ar[rd,"(F_i)_\red"',twoheadrightarrow] & (R_{i+1}/I_0R_{i+1})_\red \\
 & (R_i/I_0R_i)_\red \uar["(\ol{t_i})_\red"',hookrightarrow]
\end{tikzcd}
\]
Here the absolute Frobenius $\varphi$ is injective because $(R_{i+1}/I_0R_{i+1})_\red$ is reduced. Thus $(F_i)_\red$ is also injective, which yields the assertion.
\end{proof}

The last ingredient of our proof is the following lemma on decomposable properties of rings.

\begin{lemma}
\label{lem:BC}
Let $(A,I)$ be a pair, and let $\wt{A}\coloneqq A/A_{\Itor}$.
If $A_{\textrm{$I$-}\mathrm{tor}}$ does not contain any non-zero nilpotent element of $A$, then the commutative diagram
\[
\begin{tikzcd}
A \rar \dar & \wt{A} \dar \\
(A/I)_\red \rar & (\wt{A}/I\wt{A})_\red
\end{tikzcd}
\]
consisting of canonical projections of rings is cartesian, that is, $A_{\textrm{$I$-}\mathrm{tor}}\cap \sqrt{I} = (0)$.
\end{lemma}

\begin{proof}
Since $A_{\Itor}=A_{\textrm{$\sqrt{I}$-}\mathrm{tor}}$ and $(\wt{A}/I\wt{A})_\red = (\wt{A}/\sqrt{I}\wt{A})_\red$, we may assume that $I=\sqrt{I}$. Pick $x\in A_{\textrm{$I$-}\mathrm{tor}}\cap I$. Then there exists some $n>0$ such that $x^{n+1}=x^n\cdot x=0$. But $A_{\textrm{$I$-}\mathrm{tor}}$ does not contain any non-zero nilpotent element of $A$, and thus $x=0$.%\footnote{In fact, we can take $m=1$ because $IA_{\textrm{$I$-}\mathrm{tor}}=(0)$ by \cite[Lemma 3.18]{INS25}.}
\end{proof}

\begin{proof}[{Proof of \cref{thm:decomp}}]
(1) follows from \cref{prop:torftower}.

(2) follows from (1) and \cref{prop:modIred}.

(3) One can deduce from \textbf{(b)} and \textbf{(g)} that $(R_i)_{\Iztor}$ does not contain any non-zero nilpotent element of $R_i$ for every $i\geq 0$ (see the proof of \cite[Lemma 3.47]{INS25}). Thus the assertion follows from \cref{lem:BC} and \cref{prop:torftower} (3) (4).
\end{proof}

%%%%%%%%%%%%%%%%%%%%%%%%%%%%%%%%%%%%%%%%%%%%%%%%%%%%%%%%%%%%%%%%%%%%%%%%%%%%%%%%%%%%%%%%%%%%%%%
\subsection{Reducedness of separated perfectoid towers}

We begin with the $I_0$-torsion free case (\cref{prop:torfred}).

\begin{lemma}
\label{lem:inj lift}
Let $(A,a)$ be a pair with $a\in A$. Let $f\colon M\to N$ be a morphism of $A$-modules, and suppose that $M$ is $a$-adically separated and $N$ is $a$-torsion free. If the induced map $\ol{f}\colon M/aM\to N/aN$ is injective, then so is $f$. %Moreover, the similar assertions hold for $A$-algebras.
\end{lemma}

\begin{proof}
Pick $x\in\Ker f$. Then $\ol{f}(x\ \mathrm{mod}\ aM)=0$, and so $x\in aM$. Hence we can write $x=ax_1$ for some $x_1\in M$. Moreover, we obtain $f(x_1)=0$ because $N$ is $a$-torsion free. By induction, we see that for any $n>0$ there exists $x_n\in M$ such that $x=a^nx_n$. Since $M$ is $a$-adically separated, we obtain $x\in\bigcap_{n>0}a^nM=0$, as desired.
\end{proof}

\cref{lem:inj lift} was motivated by the following special case:

\begin{example}
Let $(A,\delta)$ be a $\delta$-ring, and let $\phi\colon A\to A$ denote the associated Frobenius lift. Then applying \cref{lem:inj lift} to $\phi\colon A\to \phi_*A$ recovers \cite[Lemma 3.7]{Ish26}: if $A$ is $p$-adiclly separated and $p$-torsion free and $A/pA$ is reduced, then $\phi$ is injective.
\end{example}

%By \cref{lem:inj lift} we immediately have the following proposition.
%
%\begin{proposition}
%\label{prop:ti inj}
%Let $(R,I_0)$ be a pair, and $\textrm{{\boldmath $R$}}=\{R_i,t_i\}_{i\geq 0}$ a tower of rings. Suppose that the following conditions are satisfied.
%  \begin{enumerate}
%  \item {\boldmath $R$} satisfies \emph{\textbf{(a)}} and \emph{\textbf{(b)}} for $(R,I_0)$.
%  \item For any $i\geq 0$, $R_i$ is $I_0$-adically separated and $I_0$-torsion free.
%  \end{enumerate}
%Then the map $t_i\colon R_i\to R_{i+1}$ is injective for any $i\geq 0$.
%\end{proposition}

Then we get the following result.

\begin{proposition}
\label{prop:torfred}
Let $\textrm{{\boldmath $R$}}=\{R_i,t_i\}_{i\geq 0}$ be an $I_0$-adically separated and $I_0$-torsion free purely inseparable tower of rings arising from a pair $(R,I_0)$.
  \begin{enumerate}
  \item The map $t_i\colon R_i\to R_{i+1}$ is injective for every $i\geq 0$.
  \item If, moreover, {\boldmath $R$} satisfies \emph{\textbf{(f)}}, then it is reduced.
  \end{enumerate}
\end{proposition}

\begin{proof}
(1) immediately follows from \cref{lem:inj lift}.

(2) By (1), it suffices to show that $R_{i+1}$ is reduced for every $i\geq 0$. By \textbf{(f)}, there exists a principal ideal $I_1\subset R_1$ such that $I_1^p=I_0R_1$ and the absolute Frobenius of $R_{i+1}/I_0R_{i+1}$ induces an injection $R_{i+1}/I_1R_{i+1} \hookrightarrow R_{i+1}/I_0R_{i+1}$. Thus the assertion follows from (a similar argument as in) \cref{lem:inj lift}.
\end{proof}

The general case can be proven with the aid of \cref{thm:decomp}; we present the details of the following observation found in \cite[footnote 7]{CS24}.

\begin{lemma}
\label{lem:non-unital}
Let $(A,I)$ be a pair.
  \begin{enumerate}
  \item $A$ is $I$-adically separated (resp.\ complete) if and only if every $I$-adic Cauchy sequence in $I$ has at most one (resp.\ a unique) limit in $I$.
  \item Let $(B,J)$ be another pair. If $I\cong J$ as (possibly) non-unital rings, then $A$ is $I$-adically separated (resp.\ complete) if and only if $B$ is $J$-adically separated (resp.\ complete).
  \end{enumerate}
\end{lemma}

\begin{proof}
(1) The ``only if'' part follows from the fact that $I$ is $I$-adically closed (or even open) in $A$. To show the converse, let $\{x_n\}_{n\geq 0}$ be a sequence in $I$ such that for $n\leq m$ we have $x_n\equiv x_m\ \mathrm{mod}\ I^n$. For each $n\geq 1$, let $y_n\coloneqq x_{n+1}-x_n$. Then $\{y_n\}_{n\geq 1}$ is a sequence in $I$ such that for $n\geq 1$ we have $y_n\equiv y_{n+1}\ \mathrm{mod}\ I^{n}$. Since $x_n=x_0+\sum\limits_{k=1}^n y_k$ for every $n\geq 0$, $\{x_n\}_{n\geq 0}$ converges if and only if so does $\{y_n\}_{n\geq 1}$, and in this case $\lim\limits_{n\to\infty}x_n=x_0+\lim\limits_{n\to\infty}y_n$. This establishes the assertion.

(2) follows from (1), which implies that the property that $A$ be $I$-adically separated (resp.\  complete) only depends on $I$ as a (possibly) non-unital ring.
\end{proof}

\begin{proposition}
\label{prop:separated}
Let $(A,I)$ be a pair such that $A_{\Itor}\cap I=(0)$. Then $A$ is $I$-adically separated (resp.\ complete) if and only if so is $A/A_{\Itor}$.
\end{proposition}

\begin{proof}
Since the canonical projection $A\to A/A_{\Itor}$ induces an isomorphism of (possibly) non-unital rings $IA\xr{\cong}I(A/A_{\Itor})$, the assertion follows from \cref{lem:non-unital}.
\end{proof}

\begin{remark}
For the ``only if'' part, we can use the fact that every discrete subgroup of a Hausdorff group is closed (\cite[Chapter III, \S2.1, Proposition 5]{BouGT1}).
\end{remark}

Now we have the following main result of this subsection.

\begin{corollary}
\label{cor:reduced}
Let $\textrm{{\boldmath $R$}}=\{R_i,t_i\}_{i\geq 0}$ be an $I_0$-adically separated preperfectoid tower arising from a pair $(R,I_0)$. %Suppose that for any $i\geq 0$, $R_i$ is $I_0$-adically separated.
Then $t_i$ is injective and $R_i$ is reduced for every $i\geq 0$.
\end{corollary}

\begin{proof}
By \cref{thm:decomp} (2) and \cref{prop:separated}, we may assume that $R_i$ is $I_0$-torsion free for every $i\geq 0$. Then we apply \cref{prop:torfred} to deduce the desired result.
\end{proof}

\begin{remark}
\label{rem:reduced}
Note that the assumption that {\boldmath $R$} is $I_0$-adically separated is automatic in either of the following cases:
  \begin{enumerate}
  \item {\boldmath $R$} is a Noetherian perfectoid tower arising from $(R,I_0)$; then every $R_i$ is an $I_0$-adically Zariskian Noetherian ring, hence $I_0$-adically separated by Krull's intersection theorem (\cite[Theorem 8.10 i)]{Mat2}, \cite[Chapter 0, Proposition 7.4.16]{FKI}).
  \item $\textrm{{\boldmath $R$}}=(A\xr{\id}A\xr{\id}A\xr{\id}\cdots)$ is a perfectoid tower arising from $(A,\varpi^p)$, where $A$ is a perfectoid ring that is $\varpi$-adically complete for a $\varpi\in A$ with $p\in\varpi^pA$.
  \end{enumerate}
\end{remark}

In general, a perfectoid tower is not reduced:

\begin{example}
Consider the polynomial ring $\Z_p[T]$, and let
\[
A\coloneqq \Z_p[T]/(T^2),\quad 
R_i\coloneqq A[p^{1/p^i}]\left[\tfrac{T}{p},\tfrac{T}{p^2},\tfrac{T}{p^3},\ldots\right] \subset A[p^{1/p^i}]\left[\tfrac{1}{p}\right]\quad (i\geq 0).
\]
For each $i\geq 0$, let $t_i\colon R_i\to R_{i+1}$ denote the ring map induced by the canonical injection $A[p^{1/p^i}][1/p]\hookrightarrow A[p^{1/p^{i+1}}][1/p]$. Then $\{R_i,t_i\}_{i\geq 0}$ is a preperfectoid tower arising from $(R_0,p)$ consisting of non-reduced rings.
\end{example}

\begin{proof}
We are going to check the conditions \textbf{(a)} $\sim$ \textbf{(d)}, \textbf{(f)}, \textbf{(g)} in \cref{def:perfectoidtower}.

\textbf{(a)} is obvious.

\textbf{(b)} $\sim$ \textbf{(d)}, \textbf{(f)} Since we know that $\{\Z_p[p^{1/p^i}]\}_{i\geq 0}$ is a perfectoid tower arising from $(\Z_p,p)$ with the first perfectoid pillar $(p^{1/p})$, it suffices to show the following claim.

\begin{claim}
For every $i\geq 0$, the canonical homomorphism $\Z_p[p^{1/p^i}] \to R_i$ induces an isomorphism $\Z_p[p^{1/p^i}]/(p) \xr{\cong} R_i/pR_i$.
\end{claim}

\begin{pfclaim}
For any $n\geq 0$, we have $T/p^n = p\cdot(T/p^{n+1}) = 0$ in $R_n/(p)$.
Since $R_i$ is generated by the system $\{T/p^n\}_{n\geq 0}$ as an $\Z_p[p^{1/p^i}]$-algebra, the composite $\Z[p^{1/p^i}] \to R_i \to R_i/pR_i$ is surjective.
Its kernel is a principal ideal generated by $(p^{1/p^i})^m$ for some $m\leq p^i$ because $(\Z_p[p^{1/p^i}],(p^{1/p^i}))$ is a DVR. If $m<p^i$, then we would have $pR_i = R_i$, i.e., $R_i/pR_i=0$. But then $0=R_i/(p,T) \cong A[p^{1/p^i}]/(p,T) \cong \F_p[X]/(X^{p^i})$, a contradiction. Hence $m=p^i$, as hoped.
\end{pfclaim}

\textbf{(g)} follows automatically, since $R_i$ is a domain for every $i\geq 0$.
\end{proof}

\subsection{Gluing of perfectoid towers}

We will show the converse of \cref{thm:decomp}. We prepare the following lemma about cartesian diagrams of rings. 
%Since products of perfectoid towers are again perfectoid towers (\cref{ex:perfectoid tower} (5)), one can obtain ``mixed'' examples of perfectoid towers by taking products. We will give slightly more complicated examples by passing to fiber products. 
  
\begin{lemma}
\label{lem:PBZar}
Consider the cartesian diagrams of rings
\[
\begin{tikzcd}
A \rar["\pi_2"] \dar["\pi_1"'] & A_2 \dar["f_2"] \\
A_1 \rar["f_1"] & B.
\end{tikzcd}
\]
Let $\fraka_1\subset A_1$, $\fraka_2\subset A_2$, $\frakb\subset B$ be ideals such that $f_1(\fraka_1)\subset \frakb$ and $f_2(\fraka_2)\subset\frakb$. Define the ideal $\fraka\subset A$ as the fiber product $\fraka_1\times_\frakb\fraka_2$ of $A$-modules.
  \begin{enumerate}
  \item The induced diagram of rings
  \[
  \begin{tikzcd}
  A/\fraka \rar["\ol{\pi_2}"] \dar["\ol{\pi_1}"'] & A_2/\fraka_2 \dar["\ol{f_2}"] \\
  A_1/\fraka_1 \rar["\ol{f_1}"'] & B/\frakb
  \end{tikzcd}
  \]
  is cartesian.
  \item If $(A_1,\fraka_1)$, $(A_2,\fraka_2)$, and $(B,\frakb)$ are Zariskian, then so is $(A,\fraka)$.\footnote{We say that a pair $(A,I)$ is \emph{Zariskian} if $A$ is $I$-adically Zariskian.}
  \end{enumerate}
\end{lemma}

\begin{proof}
(1) We have the commutative diagram of abelian groups with exact columns
\[
\begin{tikzcd}[ampersand replacement=\&]
 \& 0 \dar \& 0 \dar \&[0.5cm] 0 \dar \& \\
0 \rar \& \fraka_0 \rar \dar \& \fraka_1 \oplus \fraka_2 \rar \dar \& \frakb \rar \dar \& 0 \\
0 \rar \& A \rar["{\begin{bsmatrix}\pi_1\\ \pi_2\end{bsmatrix}}"] \dar \& A_1 \oplus A_2 \rar["{\begin{bsmatrix}f_1 & -f_2\end{bsmatrix}}"] \dar \& B \rar \dar \& 0 \\
0 \rar \& A/\fraka \rar["{\begin{bsmatrix}\ol{\pi_1}\\ \ol{\pi_2}\end{bsmatrix}}"] \dar \& A_1/\fraka_1 \oplus A_2/\fraka_2 \rar["{\begin{bsmatrix}\ol{f_1} & -\ol{f_2}\end{bsmatrix}}"] \dar \& B/\frakb \rar \dar \& 0 \\
 \& 0 \& 0 \& 0.
\end{tikzcd}
\]
Since the first and second rows are exact, the third row is exact by the nine lemma. Hence the assertion follows.

(2) follows immediately from the fact that a pair $(A,I)$ is Zariskian if and only if an element $a\in A$ is invertible if and only if $a\ \mathrm{mod}\ I$ is invertible in $A/I$ (\cite[Chapter 0, Proposition 7.3.2]{FKI}). Indeed, if $\ol{a}\coloneqq a\ \mathrm{mod}\ \fraka$ is invertible in $A/\fraka$, then for each $i\in\{1,2\}$, $\ol{\pi_i}(\ol{a})$ and $(\ol{f_i}\circ\ol{\pi_i})(\ol{a})$ are invertible in $A_i/\fraka_i$ and $B/\frakb$, respectively. Hence $a_i\coloneqq \pi_i(a)$ and $b\coloneqq f_i\circ\pi_i(a)$ are invertible in $A_i$ and $B$, respectively. Then $(a_1^{-1},a_2^{-1})\in A_1\times_BA_2=A$ gives the inverse of $a$.
\end{proof}

We now arrive at the promised construction of perfectoid towers.

\begin{proposition}
\label{prop:gluetower}
Let $\textrm{{\boldmath $R$}}'=\{R'_i,t'_i\}_{i\geq 0},\textrm{{\boldmath $S$}}=\{S_i,u_i\}_{i\geq 0}, \textrm{{\boldmath $S$}}'=\{S'_i,u'_i\}_{i\geq 0}$ be three preperfectoid towers arising from pairs $(R',I'_0)$, $(S,J_0)$, $(S',J'_0)$, with first perfectoid pillars $I'_1,J_1,J'_1$, respectively. Let $\psi'\colon \textrm{{\boldmath $R$}}'\to\textrm{{\boldmath $S$}}$ and $\sigma\colon
\textrm{{\boldmath $S$}}\to \textrm{{\boldmath $S$}}'$ be morphisms of towers such that $I'_0S'=J_0S'=J'_0$ and $I'_1S'_1=J_1S'_1=J'_1$, and suppose that $\sigma_i$ is surjective for all $i\geq 0$. Define the tower of rings $\textrm{{\boldmath $R$}}=\{R_i,t_i\}_{i\geq 0}$ (resp.\ the ideals $I_0\subset R_0$, $I_1\subset R_1$) as the fiber product in the resulting cartesian diagram of towers of rings (resp.\ of $R_0$-, $R_1$-modules)
\[
\begin{tikzcd}
\textrm{{\boldmath $R$}} \rar["\psi"] \dar["\rho"'] & \textrm{{\boldmath $S$}} \dar["\sigma"] \\
\textrm{{\boldmath $R$}}' \rar["\psi'"'] & \textrm{{\boldmath $S$}}',
\end{tikzcd}
\qquad
\begin{tikzcd}
I_0 \rar["\psi_0"] \dar["\rho_0"'] & J_0 \dar["\sigma_0"] \\
I'_0 \rar["\psi'_0"'] & J'_0,
\end{tikzcd}
\qquad
\begin{tikzcd}
I_1 \rar["\psi_1"] \dar["\rho_1"'] & J_1 \dar["\sigma_1"] \\
I'_1 \rar["\psi'_1"'] & J'_1.
\end{tikzcd}
\]
Suppose that the following conditions are satisfied.
  \begin{enumroman}
  \item $I_0$ is generated by an element $f_0=(f'_0,g_0)\in R_0$ such that $f'_0R'_0=I'_0$ and $g_0S_0=J_0$ (hence $\sigma_0(g_0)S'_0=J'_0$).
  \item $I_1$ is generated by an element $f_1=(f'_1,g_1)\in R_1$ such that $f'_1R'_1=I'_1$ and $g_1S_1=J_1$ (hence $\sigma_1(g_1)S'_1=J'_1$).
  \item For any $i\geq 0$, the diagrams of $R_i$-, $R_{i+1}$-modules
  \[
  \begin{tikzcd}
  I_0R_i \rar["\psi_i"] \dar["\rho_i"'] & J_0S_i \dar["\sigma_i"] \\
  I'_0R'_i \rar["\psi'_i"'] & J'_0S'_i,
  \end{tikzcd}
  \qquad
  \begin{tikzcd}
  I_1R_{i+1} \rar["\psi_{i+1}"] \dar["\rho_{i+1}"'] & J_1S_{i+1} \dar["\sigma_{i+1}"] \\
  I'_1R'_{i+1} \rar["\psi'_{i+1}"'] & J'_1S'_{i+1}
  \end{tikzcd}
  \]
  are cartesian.
  \item For any $i\geq 0$, the induced map $\begin{bsmatrix}(\psi'_i)_{\mathrm{tor}} & (\sigma_i)_{\mathrm{tor}}\end{bsmatrix}\colon (R'_i)_{\textrm{$I'_0$-}\mathrm{tor}}\oplus (S_i)_{\textrm{$J_0$-}\mathrm{tor}} \to (S'_i)_{\textrm{$J'_0$-}\mathrm{tor}}$ is surjective.
  \end{enumroman}
Let $R\coloneqq R_0$. Then the following hold.
  \begin{enumerate}
  \item {\boldmath $R$} is a preperfectoid tower arising from $(R,I_0)$ with first perfectoid pillar $I_1$.
  \item If {\boldmath $R'$}, {\boldmath $S$}, {\boldmath $S'$} are perfectoid towers arising from $(R',I'_0)$, $(S,J_0)$, $(S',J'_0)$, respectively, then {\boldmath $R$} is a perfectoid tower arising from $(R,I_0)$.
  \item The tilt $\textrm{{\boldmath $R$}}^\flat$ of {\boldmath $R$} is given by
  \[
  {R'}^{s.\flat}\times_{{S'}^{s.\flat}}S^{s.\flat} = {R'_0}^{s.\flat}\times_{{S'_0}^{s.\flat}}S_0^{s.\flat} \to {R'_1}^{s.\flat}\times_{{S'_1}^{s.\flat}}S_1^{s.\flat} \to \cdots \to {R'_i}^{s.\flat}\times_{{S'_i}^{s.\flat}}S_i^{s.\flat}\to\cdots.
  \]
  %\item The $I_0$-completed colimit $\wh{R_\infty}$ of the perfectoid tower {\boldmath $R$} is isomorphic to the fiber product of perfectoid rings $\wh{R'_\infty}\times_{\wh{S'_\infty}}\wh{S_\infty}$. 完備化するときに少し条件が必要そう、、
  \end{enumerate}
\end{proposition}

\begin{proof}
%To begin with, we remark:
%
%\begin{claim}
%\label{claim:pbquot}
%For any $i\geq 0$, the induced diagrams of $R_i/I_0R_i$-algebras and of $R_{i+1}/I_0R_{i+1}$-algebras
%\[
%\begin{tikzcd}%[sep=small]
%R_i/I_0R_i \rar["\ol{\psi_i}"] \dar["\ol{\rho_i}"'] & S_i/J_0S_i \dar["\ol{\sigma_i}"] \\
%R'_i/I'_0R'_i \rar["\ol{\psi'_i}"'] & S'_i/J'_0S'_i
%\end{tikzcd}
%\qquad
%\begin{tikzcd}%[sep=small]
%R_{i+1}/I_1R_{i+1} \rar["\ol{\psi_{i+1}}"] \dar["\ol{\rho_{i+1}}"'] & S_{i+1}/J_0S_{i+1} \dar["\ol{\sigma_{i+1}}"] \\
%R'_{i+1}/I'_1R'_{i+1} \rar["\ol{\psi'_{i+1}}"'] & S'_{i+1}/J'_0S'_{i+1}
%\end{tikzcd}
%\]
%are cartesian.
%\end{claim}
%
%\begin{pfclaim}
%We have the commutative diagram with exact columns
%\[
%\begin{tikzcd}[ampersand replacement=\&]
% \& 0 \dar \& 0 \dar \&[0.5cm] 0 \dar \& \\
%0 \rar \& I_0R_i \rar \dar \& I'_0R'_i \oplus J_0S_i \rar \dar \& J'_0S'_i \rar \dar \& 0 \\
%0 \rar \& R_i \rar["{\begin{bsmatrix}\rho_i\\ \psi_i\end{bsmatrix}}"] \dar \& R'_i \oplus S_i \rar["{\begin{bsmatrix}\psi'_i & -\sigma_i\end{bsmatrix}}"] \dar \& S'_i \rar \dar \& 0 \\
%0 \rar \& R_i/I_0R_i \rar["{\begin{bsmatrix}\ol{\rho_i}\\ \ol{\psi_i}\end{bsmatrix}}"] \dar \& (R'_i/I'_0R'_i) \oplus (S_i/J_0S_i) \rar["{\begin{bsmatrix}\ol{\psi'_i} & -\ol{\sigma_i}\end{bsmatrix}}"] \dar \& (S'_i/J'_0S'_i) \rar \dar \& 0 \\
% \& 0 \& 0 \& 0.
%\end{tikzcd}
%\]
%Since the first row is exact by the assumption (3) and the second row is exact by the definition of $R_i$, the third row is exact by the nine lemma. Hence the first diagram is cartesian. The proof for the second diagram is similar.
%\end{pfclaim}
(1) We are going to check the conditions \textbf{(a)} $\sim$ \textbf{(d)}, \textbf{(f)}, \textbf{(g)} in \cref{def:perfectoidtower}.

\textbf{(a)} is obvious.

\textbf{(b)} For every $i\geq 0$, we have the commutative diagram of abelian groups with exact rows by \cref{lem:PBZar} (1)
\[
\begin{tikzcd}[ampersand replacement=\&]
0 \rar \& R_i/I_0R_i \rar["{\begin{bsmatrix}\ol{\rho_i}\\ \ol{\psi_i}\end{bsmatrix}}"] \dar["\ol{t_i}"'] \&[0.5cm] (R'_i/I'_0R'_i) \oplus (S_i/J_0S_i) \rar["{\begin{bsmatrix}\ol{\psi'_i} & -\ol{\sigma_i}\end{bsmatrix}}"] \dar["\ol{t'_i}\oplus\ol{u_i}"'] \&[1.5cm] S'_i/J'_0S'_i \rar \dar["\ol{u'_i}"] \& 0 \\
0 \rar \& R_{i+1}/I_0R_{i+1} \rar["{\begin{bsmatrix}\ol{\rho_{i+1}}\\ \ol{\psi_{i+1}}\end{bsmatrix}}"] \& (R'_{i+1}/I'_0R'_{i+1}) \oplus (S_{i+1}/J_0S_{i+1}) \rar["{\begin{bsmatrix}\ol{\psi'_{i+1}} & -\ol{\sigma_{i+1}}\end{bsmatrix}}"] \rar \& S'_{i+1}/J'_0S'_{i+1} \rar \& 0.
\end{tikzcd}
\]
Since $\ol{t'_i}$ and $\ol{u_i}$ are injective, so is $\ol{t_i}$.

\textbf{(c)} For every $i\geq 0$, we have the commutative diagram of solid arrows with exact rows by \cref{lem:PBZar} (1)
\[
\begin{tikzcd}[ampersand replacement=\&]
0 \rar \& R_{i+1}/I_0R_{i+1} \rar["{\begin{bsmatrix}\ol{\rho_{i+1}}\\ \ol{\psi_{i+1}}\end{bsmatrix}}"] \dar["F_i"',dashed] \&[0.5cm] (R'_{i+1}/I_0R'_{i+1})\oplus (S_{i+1}/J_0S_{i+1}) \rar["{\begin{bsmatrix}\ol{\psi'_{i+1}} & -\ol{\sigma_{i+1}}\end{bsmatrix}}"] \dar["F'_i\oplus G_i"] \&[1.5cm] S'_{i+1}/J'_0S'_{i+1} \rar \dar["G'_i"] \& 0 \\
0 \rar \& R_i/I_0R_i \rar["{\begin{bsmatrix}\ol{\rho_i}\\ \ol{\psi_i}\end{bsmatrix}}"] \& (R'_i/I_0R'_i)\oplus (S_i/J_0S_i) \rar["{\begin{bsmatrix}\ol{\psi'_i} & -\ol{\sigma_i}\end{bsmatrix}}"] \& S'_i/J'_0S'_i \rar \& 0,
\end{tikzcd}
\]
where $F'_i$ (resp.\ $G_i$, $G'_i$) is the $i$-th Frobenius projection for $\textrm{{\boldmath $R$}}'$ (resp.\ {\boldmath $S$}, $\textrm{{\boldmath $S$}}'$). Then we have the dotted arrow  $F_i$ by the universality of kernels. Here $F_i$ can only be taken as a homomorphism of abelian groups. But $\begin{bsmatrix}\ol{\rho_i}\\ \ol{\psi_i}\end{bsmatrix}\circ F_i = (F'_i\oplus G_i)\circ \begin{bsmatrix}\ol{\rho_{i+1}}\\ \ol{\psi_{i+1}}\end{bsmatrix}$ is a ring homomorphism and $\begin{bsmatrix}\ol{\rho_i}\\ \ol{\psi_i}\end{bsmatrix}$ is injective, and thus we see that $F_i$ is a ring homomorphism. Moreover, it follows from the construction that $\ol{t_i}\circ F_i$ coincides with the absolute Frobenius of $R_{i+1}/I_0R_{i+1}$.

\textbf{(d)} and \textbf{(f-2)} Fix $i\geq 0$. We have the commutative diagram of abelian groups with exact columns
\[
\begin{tikzcd}[ampersand replacement=\&]
 \&[-0.5cm] 0 \dar \& 0 \dar \&[1cm] 0 \dar \&[-0.5cm] \\
0 \rar \& I_1(R_{i+1}/I_0R_{i+1}) \rar \dar \& I_1(R'_{i+1}/I_0R'_{i+1})\oplus J_1(S_{i+1}/J_0S_{j+1}) \rar \dar \& J'_1(S'_{i+1}/J'_0S'_{i+1}) \rar \dar \& 0 \\
0 \rar \& R_{i+1}/I_0R_{i+1} \rar["{\begin{bsmatrix}\ol{\rho_{i+1}}\\ \ol{\psi_{i+1}}\end{bsmatrix}}"] \dar \& (R'_{i+1}/I_0R'_{i+1})\oplus (S_{i+1}/J_0S_{i+1}) \rar["{\begin{bsmatrix}\ol{\psi'_{i+1}} & -\ol{\sigma_{i+1}}\end{bsmatrix}}"] \dar \& S'_{i+1}/J'_0S'_{i+1} \rar \dar \& 0 \\
0 \rar \& R_{i+1}/I_1R_{i+1} \rar["{\begin{bsmatrix}\ol{\rho_{i+1}}\\ \ol{\psi_{i+1}}\end{bsmatrix}}"] \dar \& (R'_{i+1}/I_1R'_{i+1})\oplus (S_{i+1}/J_1S_{i+1}) \rar["{\begin{bsmatrix}\ol{\psi'_{i+1}} & -\ol{\sigma_{i+1}}\end{bsmatrix}}"] \dar \& S'_{i+1}/J'_1S'_{i+1} \rar \dar \& 0 \\
 \& 0 \& 0 \& 0.
\end{tikzcd}
\]
Since the second and third rows are exact by \cref{lem:PBZar} (1), so is the first row by the nine lemma. Hence we have the commutative diagram of abelian groups with exact rows
\[
\begin{tikzcd}[ampersand replacement=\&]
 \&[-0.5cm] 0 \dar \& 0 \dar \&[1cm] 0 \dar \&[-0.5cm] \\
0 \rar \& I_1(R_{i+1}/I_0R_{i+1}) \rar \dar \& I_1(R'_{i+1}/I_0R'_{i+1})\oplus J_1(S_{i+1}/J_0S_{j+1}) \rar \dar \& J'_1(S'_{i+1}/J'_0S'_{i+1}) \rar \dar \& 0 \\
0 \rar \& R_{i+1}/I_0R_{i+1} \rar["{\begin{bsmatrix}\ol{\rho_{i+1}}\\ \ol{\psi_{i+1}}\end{bsmatrix}}"] \dar["F_i"'] \& (R'_{i+1}/I_0R'_{i+1})\oplus (S_{i+1}/J_0S_{i+1}) \rar["{\begin{bsmatrix}\ol{\psi'_{i+1}} & -\ol{\sigma_{i+1}}\end{bsmatrix}}"] \dar["F'_i\oplus G_i"] \& S'_{i+1}/J'_0S'_{i+1} \rar \dar["G'_i"] \& 0 \\
0 \rar \& R_i/I_0R_i \rar["{\begin{bsmatrix}\ol{\rho_i}\\ \ol{\psi_i}\end{bsmatrix}}"] \dar \& (R'_i/I_0R'_i)\oplus (S_i/J_0S_i) \rar["{\begin{bsmatrix}\ol{\psi'_i} & -\ol{\sigma_i}\end{bsmatrix}}"] \dar \& S'_i/J'_0S'_i \rar \dar \& 0 \\
 \& 0 \& 0 \& 0.
\end{tikzcd}
\]
Since the second and third columns are exact, so is the first column by the nine lemma. This establishes \textbf{(d)} as well as \textbf{(f-2)}.

\textbf{(f-1)} Let $f_1=(f'_1,g_1)\in R_1$ be a generator of $I_1$ as in the assumption (ii). Then we have the commutative diagram of $R_1$-modules with exact rows by the assumption (iii)
\[
\begin{tikzcd}[ampersand replacement=\&]
0 \rar \& R_1 \rar["{\begin{bsmatrix}\rho_1\\ \psi_1\end{bsmatrix}}"] \dar["f_1^p"'] \& R'_1 \oplus S_1 \rar["{\begin{bsmatrix}\psi'_1 & -\sigma_1\end{bsmatrix}}"] \dar["{f'_1}^p\oplus g_1^p",twoheadrightarrow] \&[0.3cm] S'_1 \rar \dar["{g'_1}^p",twoheadrightarrow] \& 0 \\
0 \rar \& I_0R_1 \rar["{\begin{bsmatrix}\rho_1\\ \psi_1\end{bsmatrix}}"] \& I'_0R'_1\oplus J_0S_1 \rar["{\begin{bsmatrix}\psi'_1 & -\sigma_1\end{bsmatrix}}"] \& J'_0S'_1 \rar \& 0,
\end{tikzcd}
\]
and we have to show that the first vertical map is surjective. By the snake lemma, it suffices to show that the induced map
\[
\begin{bsmatrix}(\psi'_1)_{\mathrm{tor}} & (\sigma_1)_{\mathrm{tor}}\end{bsmatrix} \colon (R'_1)_{\textrm{$I'_0$-}\mathrm{tor}} \oplus (S_1)_{\textrm{$J_0$-}\mathrm{tor}} \to (S'_1)_{\textrm{$J'_0$-}\mathrm{tor}}
\]
is surjective, which follows from the assumption (iv).

\textbf{(g)} Fix $i\geq 0$. Take a generator $f_0=(f'_0,g_0)\in R$ of $I_0$ as in the assumption (i). Then we have the commutative diagram of $R_i$-modules with exact rows by the assumption (iv)
\[
\begin{tikzcd}[ampersand replacement=\&]
0\rar \& (R_i)_{\Iztor} \rar["{\begin{bsmatrix}(\rho_i)_{\mathrm{tor}}\\ (\psi_i)_{\mathrm{tor}}\end{bsmatrix}}"] \dar["f_0"'] \&[0.6cm] (R'_i)_{\textrm{$I'_0$-}\mathrm{tor}} \oplus (S_i)_{\textrm{$J_0$-}\mathrm{tor}} \rar["{\begin{bsmatrix}(\psi'_i)_{\mathrm{tor}} & -(\sigma_i)_{\mathrm{tor}}\end{bsmatrix}}"] \dar["f'_0\oplus g_0"] \&[1.8cm] (S'_i)_{\textrm{$J'_0$-}\mathrm{tor}}\rar \dar["g'_0"] \& 0 \\
0\rar \& (R_i)_{\Iztor} \rar["{\begin{bsmatrix}(\rho_i)_{\mathrm{tor}}\\ (\psi_i)_{\mathrm{tor}}\end{bsmatrix}}"] \& (R'_i)_{\textrm{$I'_0$-}\mathrm{tor}} \oplus (S_i)_{\textrm{$J_0$-}\mathrm{tor}} \rar["{\begin{bsmatrix}(\psi'_i)_{\mathrm{tor}} & -(\sigma_i)_{\mathrm{tor}}\end{bsmatrix}}"] \& (S'_i)_{\textrm{$J'_0$-}\mathrm{tor}}\rar \& 0.
\end{tikzcd}
\]
Since the second vertical map is zero, so is the first vertical map. Hence $I_0(R_i)_{\Iztor}=(0)$.

Moreover, we have the commutative diagram of abelian groups with exact rows by the assumption (iv)
\[
\begin{tikzcd}[ampersand replacement=\&]
0\rar \& (R_{i+1})_{\Iztor} \rar["{\begin{bsmatrix}(\rho_{i+1})_{\mathrm{tor}}\\ (\psi_{i+1})_{\mathrm{tor}}\end{bsmatrix}}"] \dar["(F_i)_{\mathrm{tor}}"',dashed,"\cong"] \&[0.9cm] (R'_{i+1})_{\textrm{$I'_0$-}\mathrm{tor}} \oplus (S_{i+1})_{\textrm{$J_0$-}\mathrm{tor}} \rar["{\begin{bsmatrix}(\psi'_{i+1})_{\mathrm{tor}}& -(\sigma_{i+1})_{\mathrm{tor}}\end{bsmatrix}}"] \dar["(F'_i)_{\mathrm{tor}}\oplus (G_i)_{\mathrm{tor}}","\cong"'] \&[2.3cm] (S'_{i+1})_{\textrm{$J'_0$-}\mathrm{tor}} \rar \dar["\cong"',"(G'_i)_{\mathrm{tor}}"] \& 0 \\
0\rar \& (R_i)_{\Iztor} \rar["{\begin{bsmatrix}(\rho_i)_{\mathrm{tor}}\\ (\psi_i)_{\mathrm{tor}}\end{bsmatrix}}"] \& (R'_i)_{\textrm{$I'_0$-}\mathrm{tor}} \oplus (S_i)_{\textrm{$J_0$-}\mathrm{tor}} \rar["{\begin{bsmatrix}(\psi'_i)_{\mathrm{tor}} & -(\sigma_i)_{\mathrm{tor}}\end{bsmatrix}}"] \& (S'_i)_{\textrm{$J'_0$-}\mathrm{tor}}\rar \& 0,
\end{tikzcd}
\]
where $(F_i)_{\mathrm{tor}}$ is the unique map induced by the universal property of kernels. This completes the proof of (1).

(2) follows from (1) and \cref{lem:PBZar} (2).

(3) follows from the construction of the Frobenius projections $F_i$ for {\boldmath $R$}; note that the projective system $\{R_i/I_0R_i,F_i\}_{i\geq 0}$ is strict.
\end{proof}

One of the most interesting examples of \cref{prop:gluetower} is the following.

\begin{remark}%[{Modifications in characteristic $p$; \cite[Example 3.8 (6)]{BIM19}, \cite[Example 2.1.5]{CS24}}]
\label{rem:gluetower}
The assumption of \cref{prop:gluetower} is satisfied if $\textrm{{\boldmath $S$}}'$ is a perfect tower and either $\textrm{{\boldmath $R$}}'$ or {\boldmath $S$} is a perfect tower (i.e., $J'_0=(0)$ and either $I'_0=(0)$ or $J_0=(0)$). 
Hence one can construct new (pre)perfectoid towers from old ones as follows: 
if $\textrm{{\boldmath $S$}}=\{S_i,u_i\}_{i\geq 0}$ is a (pre)perfectoid tower arising from a pair $(S,J_0)$, and $R'\to (S/J_0S)_\red$ is any homomorphism of reduced $\F_p$-algebras, then
\[
R\coloneqq R'\times_{(S/J_0S)_\red} S \to (R')^{\frac{1}{p}}\times_{(S_1/J_0S_1)_\red} S_1 \to \cdots \to (R')^{\frac{1}{p^i}}\times_{(S_i/J_0S_i)_\red} S_i\to \cdots
\]
is a (pre)perfectoid tower arising from $(R,I_0)$, where $I_0=(0)R'\times_{(0)(S/J_0S)_\red}J_0$ (the fact that $\{(S_i/J_0S_i)_\red\}_{i\geq 0}$ is a perfect tower follows from \cref{prop:modIred}).
\end{remark}

%Using \cref{prop:gluetower}, one can construct examples of perfectoid towers that have $p$-torsion:

\begin{example}
\label{ex:PB}
Consider the perfectoid tower $\Z_p\hookrightarrow\Z_p[p^{1/p}] \hookrightarrow \cdots \hookrightarrow \Z_p[p^{1/p^i}]\hookrightarrow\cdots$ arising from $(\Z_p,p)$ and the canonical projection $\F_p\llbracket x,y\rrbracket\to\F_p$. Then for each $i\geq 0$ we have the cartesian diagram of rings
\[
\begin{tikzcd}
A[p^{1/p^i},x^{1/p^i},y^{1/p^i}]/(p^{1/p^i}x^{1/p^i},p^{1/p^i}y^{1/p^i}) \rar \dar \ar[rd,phantom,"\textrm{p.b.}"] & \Z_p[p^{1/p^i}] \dar \\
\F_p\llbracket x,y\rrbracket[x^{1/p^i},y^{1/p^i}] \rar & \F_p,
\end{tikzcd}
\]
where $A\coloneqq\Z_p\llbracket x,y\rrbracket/(px,py)$. Hence $\{A[p^{1/p^i},x^{1/p^i},y^{1/p^i}]/(p^{1/p^i}x^{1/p^i},p^{1/p^i}y^{1/p^i})\}_{i\geq 0}$ is a perfectoid tower arising from $(A,p)$.

This example was already found in Ishiro--Shimomoto \cite[Example 4.27]{IS25}. More generally, they gave a construction of perfectoid towers arising from quotients by squarefree monomial ideals with $p$-torsion elements. While their proof is based on a calculation, \cref{prop:gluetower} gives a conceptual proof of it. Indeed, in the notations as in \cite[\S4.4]{IS25}, verify that the diagram of rings
\[
\begin{tikzcd}
A[x_1^{1/p^i},\ldots,x_d^{1/p^d}]/(\mathbf{x}_1^{1/p^i},\ldots,\mathbf{x}_n^{1/p^i}) \rar \dar & A[x_1^{1/p^i},\ldots,x_d^{1/p^i}]/(\mathbf{y}_1^{1/p^i},\ldots,\mathbf{y}_l^{1/p^i}) \dar \\
k\llbracket\ol{p^{1/p^i}},\ol{x_2^{1/p^i}},\ldots,\ol{x_d^{1/p^i}}\rrbracket/(\ol{\mathbf{x}_1^{1/p^i}},\ldots,\ol{\mathbf{x}_n^{1/p^i}}) \rar & k\llbracket\ol{p^{1/p^i}},\ol{x_2^{1/p^i}},\ldots,\ol{x_d^{1/p^i}}\rrbracket/(\mathbf{y}_1^{1/p^i},\ldots,\mathbf{y}_l^{1/p^i}))
\end{tikzcd}
\]
is cartesian, and apply the result of the $p$-torsion free case \cite[\S4.1]{IS25} to conclude that
\[
\{A[x_1^{1/p^i},\ldots,x_d^{1/p^i}]/(\mathbf{y}_1^{1/p^i},\ldots,\mathbf{y}_l^{1/p^i})\}_{i\geq 0}
\]
is a perfectoid tower arising from $(A,p)$.
\end{example}

%%%%%%%%%%%%%%%%%%%%%%%%%%%%%%%%%%%%%%%%%%%%%%%%%%%%%%%%%%%%%%%%%%%%%%%%%%%%%%%%%%%%%%%%%%%%%%%
\subsection{Stability of preperfectoid towers under weakly \'{e}tale base change}
\label{ss:wetbc}

In this subsection, we consider a certain base change stability of preperfectoid towers.

For a not necessarily unital $\F_p$-algebra $A$, let $F_*A$ denote the ring $A$ considered as an $A$-module via restriction of scalars for the absolute Frobenius.\footnote{For more details of the $A$-algebra $F_*A$, we refer to \cite{SS}.}

\begin{lemma}
Let $\textrm{{\boldmath $R$}}=\{R_i,t_i\}_{i\geq 0}$ be a purely inseparable tower arising from a pair $(R,I_0)$, and fix $i\geq 0$.
  \begin{enumerate}
  \item The $i$-th Frobenius projection $F_i\colon R_{i+1}/I_0R_{i+1}\to F_*(R_i/I_0R_i)$ is $R_i/I_0R_i$-linear.
  \item Assume that {\boldmath $R$} satisfies \emph{\textbf{(g)}}. Then the map $(F_i)_{\mathrm{tor}}\colon(R_{i+1})_{\Iztor}\to F_*(R_i)_{\Iztor}$ is $R_i/I_0R_i$-linear.
  \end{enumerate}
\end{lemma}

\begin{proof}
(1) If $a\in R_i/I_0R_i$ and $x\in R_{i+1}/I_0R_{i+1}$, then $F_i(a\cdot x)=F_i(\ol{t_i}(a)x)=a^pF_i(x)=a\cdot F_i(x)$ in $F_*(R_i/I_0R_i)$. 

(2) is shown similarly to (1). %If $\ol{a}=a\ \mathrm{mod}\ I_0R_i \in\ol{R_i}$ and $x\in(R_{i+1})_{\Iztor}$, then $(F_i)_{\mathrm{tor}}(\ol{a}\cdot x)=(F_i)_{\mathrm{tor}}(t_i(a)x) = a^p (F_i)_{\mathrm{tor}}(x) = a\cdot (F_i)_{\mathrm{tor}}(x)$ in $F_*(R_i)_{\Iztor}$.
\end{proof}

For a homomorphism of $\F_p$-algebras $A\to B$, we denote the \emph{relative Frobenius} by
\[
\varphi_{B/A} \colon F_*A\otimes_AB \to F_*B.
\]

\begin{remark}
\label{rem:relFrobFun}
  \begin{enumerate}
  \item The relative Frobenius $\varphi_{B/A}$ is functorial in the following sense. Let
\begin{equation}
\label{eq:relFrobFun}
\begin{tikzcd}
A \rar["f"] \dar & A' \dar \\
B \rar["g"'] & B'
\end{tikzcd}
\end{equation}
be a commutative diagram of $\F_p$-algebras. Then we have the commutative diagram of $B$-algebras
\[
\begin{tikzcd}
F_*A\otimes_AB \rar["\varphi_{B/A}"] \dar["F_*f\otimes g"'] & F_*B \dar["F_*g"] \\
F_*A'\otimes_{A'}B' \rar["\varphi_{B'/A'}"'] & F_*B'.
\end{tikzcd}
\]
If \eqref{eq:relFrobFun} is cocartesian (i.e., $B'=A'\otimes_AB$), then $F_*A'\otimes_{A'}B' \xr{\cong} F_*A'\otimes_AB$, and thus we get the commutative diagram of $B$-algebras
\[
\begin{tikzcd}
F_*A\otimes_AB \rar["\varphi_{B/A}"] \dar["F_*f\otimes_AB"'] & F_*B \dar["F_*g"] \\
F_*A'\otimes_AB \rar["\varphi_{B'/A'}"'] & F_*B'.
\end{tikzcd}
\]
  \item We continue with the assumption that \eqref{eq:relFrobFun} is cocartesian. Then the factorization of the absolute Frobenius of $B'$ by the relative Frobenius $\varphi_{B'/A'}$ is as follows:
\[
\begin{tikzcd}
B' \rar["\varphi_{B'}"] \dar["\varphi_{A'}\otimes_AB"'] & F_*B' \\
F_*A'\otimes_AB \rar["\cong"] & F_*A'\otimes_{A'}B' \uar["\varphi_{B'/A'}"']
\end{tikzcd}
\]
  \end{enumerate}
\end{remark}

Using these notions, we have the following result on purely inseparable towers.

\begin{lemma}
\label{lem:insepBC}
Let $\textrm{{\boldmath $R$}}=\{R_i,t_i\}_{i\geq 0}$ be a purely inseparable tower arising from a pair $(R,I)$, and let $R\to S$ be a flat homomorphism of rings. Then $\textrm{{\boldmath $R$}}\otimes_RS=\{R_i\otimes_RS, t_i\otimes_RS\}_{i\geq 0}$ is a purely inseparable tower arising from $(S,IS)$. Moreover, for any $i\geq 0$, the $i$-th Frobenius projection is given by the composite
\[
\ol{S_{i+1}}=\ol{R_{i+1}}\otimes_{\ol{R_i}}\ol{S_i} \xr{F_i\otimes_{\ol{R_i}}\ol{S_i}} F_*\ol{R_i}\otimes_{\ol{R_i}}\ol{S_i} \xr{\varphi_{\ol{S_i}/\ol{R_i}}} F_*\ol{S_i},
\]
where $\ol{R_i}\coloneqq R_i/IR_i$, $S_i\coloneqq R_i\otimes_RS$, $\ol{S_i}\coloneqq S_i/IS_i = \ol{R_i}\otimes_RS$ for each $i\geq 0$.
\end{lemma}

\begin{proof}
We are going to check conditions \textbf{(a)}, \textbf{(b)}, and \textbf{(c)} in \cref{def:perfectoidtower}.

\textbf{(a)} is obvious.

\textbf{(b)} follows from the flatness of $R\to S$.

\textbf{(c)} Let $i\geq 0$. %and consider the commutative diagram of $\ol{R_i}$-modules
%\[
%\begin{tikzcd}[column sep=large]
%\ol{R_{i+1}} \rar["\varphi_{\ol{R_{i+1}}}"] \ar[rd,"F_i"' description] & F_*\ol{R_{i+1}} \\
%%\ol{R_i} \rar["\varphi_{\ol{R_i}}"'] \uar[hookrightarrow,"\ol{t_i}"]
% & F_*\ol{R_i}. \uar["F_*\ol{t_i}"',hookrightarrow]
%\end{tikzcd}
%\]
%Applying $-\otimes_{\ol{R_i}}\ol{S_i}=-\otimes_RS=-\otimes_{\ol{R_{i+1}}}\ol{S_{i+1}}$, we have the commutative diagram of $\ol{S_i}$-modules
%\[
%\begin{tikzcd}[column sep=huge]
%\ol{R_{i+1}}\otimes_{\ol{R_i}}\ol{S_i} \rar["\varphi_{\ol{R_{i+1}}}\otimes 1"] \ar[rr,bend left=20,"\varphi_{\ol{S_{i+1}}}"] \ar[rd,"F_i\otimes S" description] & (F_*\ol{R_{i+1}})\otimes_{\ol{R_i}}\ol{S_i} \rar["\varphi_{\ol{S_{i+1}}/\ol{R_{i+1}}}"] & F_*(\ol{R_{i+1}}\otimes_RS) \\
%%\ol{R_i}\otimes_{\ol{R_i}}\ol{S_i} \rar["\varphi_{\ol{R_i}}\otimes 1"'] \ar[rr,bend right=20,"\varphi_{\ol{S_i}}"'] \uar[hookrightarrow,"\ol{t_i}\otimes 1"] 
% & (F_*\ol{R_i})\otimes_{\ol{R_i}}\ol{S_i} \rar["\varphi_{\ol{S_i}/\ol{R_i}}"'] \uar[hookrightarrow,"(F_*\ol{t_i})\otimes S"'] & F_*(\ol{R_i}\otimes_RS). \uar[hookrightarrow,"F_*(\ol{t_i}\otimes S)"']
%\end{tikzcd}
%\]
%which induces the commutative diagram of rings
Since the diagram of rings
\[
\begin{tikzcd}[column sep=large]
\ol{R_i}\rar["\ol{t_i}"] \dar & \ol{R_{i+1}} \dar \\
\ol{S_i}\rar["\ol{t_i}\otimes_{\ol{R_i}}\ol{S_i}"'] & \ol{S_{i+1}}
\end{tikzcd}
\]
is cocartesian, the factorization of the absolute Frobenius of $\varphi_{\ol{S_{i+1}}}$ by the relative Frobenius $\varphi_{\ol{S_{i+1}}/\ol{R_{i+1}}}$ is as follows (\cref{rem:relFrobFun} (2)):
\[
\begin{tikzcd}
\ol{S_{i+1}} \rar["\varphi_{\ol{S_{i+1}}}"] \dar["\varphi_{\ol{R_{i+1}}}\otimes_{\ol{R_i}}\ol{S_i}"'] & F_*\ol{S_{i+1}} \\
F_*\ol{R_{i+1}}\otimes_{\ol{R_i}}\ol{S_i} \rar["\cong"] & F_*\ol{R_{i+1}}\otimes_{\ol{R_{i+1}}}\ol{S_{i+1}} \uar["\varphi_{\ol{S_{i+1}}/\ol{R_{i+1}}}"']
\end{tikzcd}
\]
Moreover, this factorization satisfies the following compatibility (\cref{rem:relFrobFun} (1))
\[
\begin{tikzcd}
\ol{S_{i+1}} \rar["\varphi_{\ol{R_{i+1}}}\otimes_{\ol{R_i}}\ol{S_i}"] \ar[rd,"F_i\otimes_{\ol{R_i}}\ol{S_i}"'] &[1cm] F_*\ol{R_{i+1}}\otimes_{\ol{R_i}}\ol{S_i} \rar["\cong"] & F_*\ol{R_{i+1}}\otimes_{\ol{R_{i+1}}}\ol{S_{i+1}} \rar["\varphi_{\ol{S_{i+1}}/\ol{R_{i+1}}}"] &[1cm] F_*\ol{S_{i+1}} \\
 & F_*\ol{R_i}\otimes_{\ol{R_i}}\ol{S_i} \uar["F_*\ol{t_i}\otimes_{\ol{R_i}}\ol{S_i}"'] \rar[equal] & F_*\ol{R_i}\otimes_{\ol{R_i}}\ol{S_i} \uar["F_*\ol{t_i}\otimes(\ol{t_i}\otimes_{\ol{R_i}}\ol{S_i})"'] \rar["\varphi_{\ol{S_i}/\ol{R_i}}"'] & F_*\ol{S_i}. \uar["F_*(\ol{t_i}\otimes_{\ol{R_i}}\ol{S_i})"']
\end{tikzcd}
\]
%\[
%\begin{tikzcd}
%\ol{S_i}\otimes_{\ol{R_i}}F_*\ol{R_i} \rar["\varphi_{\ol{S_i}/\ol{R_i}}"] \dar["(\ol{t_i}\otimes_RS)\otimes F_*\ol{t_i}"'] &[1.5cm] F_*\ol{S_i} \dar["F_*(\ol{t_i}\otimes_RS)"] \\
%\ol{S_{i+1}}\otimes_{\ol{R_{i+1}}}F_*\ol{R_{i+1}} \rar["\varphi_{\ol{S_{i+1}}/\ol{R_{i+1}}}"'] & F_*\ol{S_{i+1}}
%\end{tikzcd}
%\]
Hence we obtain the desired compatibility
\[
\begin{tikzcd}[column sep=huge]
\ol{S_{i+1}} \rar["\varphi_{\ol{S_{i+1}}}"] \ar[rd,"\varphi_{\ol{S_i}/\ol{R_i}}\circ (F_i\otimes_{\ol{R_i}}\ol{S_i})"'] & \ol{S_{i+1}} \\
 & \ol{S_i}. \uar["\ol{t_i}\otimes_{\ol{R_i}}\ol{S_i}"',hookrightarrow]
\end{tikzcd}
\]
\end{proof}

In order to obtain the base change stability of preperfectoid towers, we need the notion of \emph{weakly \'{e}tale morphisms} (a.k.a. \emph{absolutely flat morphisms}). Here we follow the terminology introduced in \cite[Definition 3.1.1]{GR} (see also \cite[Remark 2.3.2]{BS15}).

\begin{definition}
We say that a ring homomorphism $A\to B$ is called \emph{weakly \'{e}tale} if it is flat and the multiplication map $B\otimes_AB\to B$ is flat.
\end{definition}

\begin{remark}
\label{rem:wet}
The property ``weakly \'{e}tale'' is stable under %compositions and 
base change (\cite[Lemma 3.1.2 (i)]{GR}).%(\cite[\href{https://stacks.math.columbia.edu/tag/094T}{Tag 094T} and \href{https://stacks.math.columbia.edu/tag/094U}{Tag 094U}]{stacks-project}).
\end{remark}

\begin{example}
\label{ex:wet}
We have the implications
\begin{center}
\'{e}tale $\Longrightarrow$ ind-\'{e}tale $\Longrightarrow$ weakly \'{e}tale $\Longrightarrow$ formally \'{e}tale.
\end{center}
See \cite[Proposition 2.3.3]{BS15} for the proof. It is, moreover,  known that for any weakly \'{e}tale homomorphism $f\colon A\to B$, there exists a faithfully flat ind-\'{e}tale homomorphism $g\colon B\to C$ such that $g\circ f\colon A\to C$ is ind-\'{e}tale (\cite[Theorem 2.3.4]{BS15}).
\end{example}

An important property of weakly \'{e}tale morphisms is the following.

\begin{theorem}[{\cite[Theorem 3.5.13]{GR}, \cite[\href{https://stacks.math.columbia.edu/tag/0F6W}{Tag 0F6W}]{stacks-project}}]
\label{thm:wetFrob}
If $A\to B$ is a weakly \'{e}tale homomorphism of $\F_p$-algebras, then the relative Frobenius $\varphi_{B/A}$ is an isomorphism.
\end{theorem}

Moreover, we have the following lemma.

\begin{lemma}
\label{lem:wetred}
If $A\to B$ is weakly \'{e}tale, then $A_\red \otimes_AB \to B_\red$ is an isomorphism.
\end{lemma}

\begin{proof}
If $A\to B$ is weakly \'{e}tale, so is its base change $A_\red \to A_\red \otimes_AB$. Then $A_\red\otimes_AB$ is reduced by \cite[\href{https://stacks.math.columbia.edu/tag/092I}{Tag 092I}]{stacks-project}, and so we obtain the desired inverse $B_\red \to A_\red\otimes_AB$ by the universality of $B_\red$.
\end{proof}

Now we formulate our base change theorem.

\begin{theorem}
\label{thm:bc}
Let $\textrm{{\boldmath $R$}}=\{R_i,t_i\}_{i\geq 0}$ be a preperfectiod tower arising from a pair $(R,I_0)$, and let $R\to S$ be a weakly \'{e}tale homomorphism of rings.
  \begin{enumerate}
  \item $\textrm{{\boldmath $R$}}\otimes_RS=\{R_i\otimes_RS, t_i\otimes_RS\}_{i\geq 0}$ is a preperfectoid tower arising from $(S,I_0S)$.
  \item The $i$-th perfectoid pillar of $\textrm{{\boldmath $R$}}\otimes_RS$ is given by $I_iS_i$, where $I_i\subset R_i$ is the $i$-th perfectoid pillar of {\boldmath $R$}. Moreover, the small tilt of $I_iS_i$ is $I_i^{s.\flat}S_i^{s.\flat}$, where $I_i^{s.\flat}\subset R_i^{s.\flat}$ is the small tilt of $I_i$.
  \end{enumerate}
\end{theorem}

\begin{proof}
(1) \textsc{Step 1}. Let $\wt{R_i}\coloneqq R_i/(R_i)_{\Iztor}$ and $\wt{S_i}\coloneqq S_i/(S_i)_{\Iztor}$ for each $i\geq 0$. Then we have the canonical decomposition (\cref{lem:BC})
\[
R_i\xr{\cong} \wt{R_i} \times_{(\wt{R_i}/I_0\wt{R_i})_\red} (R_i/I_0R_i)_\red.
\]
Since $R_i\to S_i$ is flat, we also have the decomposition
\begin{align*}
S_i &\xr{\cong} (\wt{R_i} \otimes_{R_i}S_i) \mathop{\times}_{(\wt{R_i} /I_0\wt{R_i})_\red \otimes_{R_i}S_i} ((R_i/I_0R_i)_\red\otimes_{R_i}S_i) \\
&\xr{\cong} \wt{S_i}\times_{(\wt{S_i}/I_0\wt{S_i})_\red} (S_i/I_0S_i)_\red
\end{align*}
where the second isomorphism follows from the flatness of $R_i\to S_i$ and \cref{lem:wetred}.
Thus \cref{prop:gluetower} allows us to pass to the $I_0$-torsion free perfectoid tower 
$\wt{\textrm{{\boldmath $R$}}}=\{\wt{R_i}\}_{i\geq 0}$ and the perfect tower $(\textrm{{\boldmath $R$}}/I_0\textrm{{\boldmath $R$}})_\red=\{(R_i/I_0R_i)_{\red}\}_{i\geq 0}$.

\textsc{Step 2}. The case where {\boldmath $R$} is a perfect tower. Then we see from \cref{thm:wetFrob} that $\textrm{{\boldmath $R$}}\otimes_RS$ is a perfect tower.

\textsc{Step 3}. The case where {\boldmath $R$} is $I_0$-torsion free. By \cref{lem:insepBC}, we know that $\textrm{{\boldmath $R$}}\otimes_RS$ is a purely inseparable tower of rings arising from $(S,I_0S)$, and the $i$-th Frobenius projection is given by the composite
\[
\ol{S_{i+1}}=\ol{R_{i+1}}\otimes_{\ol{R_i}}\ol{S_i} \xr{F_i\otimes_{\ol{R_i}}\ol{S_i}} F_*\ol{R_i}\otimes_{\ol{R_i}}\ol{S_i} \xr[\cong]{\varphi_{\ol{S_i}/\ol{R_i}}} F_*\ol{S_i},
\]
where $\ol{R_i}\coloneqq R_i/I_0R_i$, $S_i\coloneqq R_i\otimes_RS$, $\ol{S_i}\coloneqq S_i/I_0S_i = \ol{R_i}\otimes_RS$. Here the relative Frobenius $\varphi_{\ol{S_i}/\ol{R_i}}$ is an isomorphism because $R\to S$ is weakly \'{e}tale (\cref{rem:wet,thm:wetFrob}). We are going to the remaining conditions \textbf{(d)}, \textbf{(f)}, and \textbf{(g)} in \cref{def:perfectoidtower}.

\textbf{(d)} Since $F_i$ is surjective, so is $F_i\otimes_{\ol{R_i}}\ol{S_i}$.

\textbf{(f)} Since $I_0\subset R$ is principal, so is $I_0S\subset S$. Let $I_1\subset R_1$ be the principal ideal as in \textbf{(f)} for {\boldmath $R$}, and consider the principal ideal $I_1S_1=I_1\otimes_RS \subset S_1 = R_1\otimes_RS$.
  \begin{enumerate}
  \item[\textbf{(f-1)}] By extending the equation $I_1^p=I_0R_1$ to $S_1$, we have $I_1^pS_1 = (I_0R_1)S_1$, that is, $(I_1S_1)^p = (I_0S)S_1$.
  \item[\textbf{(f-2)}] Applying $-\otimes_{\ol{R_i}}\ol{S_i} = -\otimes_RS$ to the equation $\Ker(F_i)=I_1\ol{R_i}$, we obtain $\Ker(F_i\otimes_RS) = I_1\ol{S_i}$.
  \end{enumerate}

\textbf{(g)} follows automatically, since we are assumed that $R_i$ (hence $S_i$) is -$I_0$torsion free for every $i\geq 0$.

(2) By the proof of (1), the first perfectoid pillar of $\textrm{{\boldmath $R$}}\otimes_RS$ is $I_1S_1$. Hence the assertion follows form \cref{lem:alphaflat}.
\end{proof}

%%%%%%%%%%%%%%%%%%%%%%%%%%%%%%%%%%%%%%%%%%%%%%%%%%%%%%%%%%%%%%%%%%%%%%%%%%%%%%%%%%%%%%%%%%%%%%%
%%%%%%%%%%%%%%%%%%%%%%%%%%%%%%%%%%%%%%%%%%%%%%%%%%%%%%%%%%%%%%%%%%%%%%%%%%%%%%%%%%%%%%%%%%%%%%%
\section{Applications to tilting correspondences}
\label{s:tilt}

In this section, we apply the results obtained in the previous section to show the invariance of several properties of perfectoid towers under tilting.

%%%%%%%%%%%%%%%%%%%%%%%%%%%%%%%%%%%%%%%%%%%%%%%%%%%%%%%%%%%%%%%%%%%%%%%%%%%%%%%%%%%%%%%%%%%%%%%
\subsection{Tilting \'{e}tale cohomology}
\label{ss:etale}

The first tilting invariants are idempotents, that is, clopen subschemes.

\begin{proposition}
\label{prop:tilting-idem}
Let $\textrm{{\boldmath $R$}}=\{R_i,t_i\}_{i\geq 0}$ be a perfectoid tower arising from a pair $(R,I_0)$.
Suppose that $R$ is $I_0$-adically henselian.
Let
\[
\Spec(R)\setminus V(I_0) \subset U\subset \Spec(R),\quad \Spec(R^{s.\flat})\setminus V(I_0^{s.\flat})\subset U^{s.\flat}\subset\Spec(R^{s.\flat})
\]
be open sets whose complements agree via the isomorphism $R^{s.\flat}/I_0^{s.\flat}\xr{\cong}R/I_0$.
Then there is a bijection of idempotents
\[
\mathrm{Idem}(U) \cong \mathrm{Idem}(U^{s.\flat}),
\]
compatibly with orthogonality, functorial in {\boldmath $R$} and $U$.
\end{proposition}
  
\begin{proof}
By \cite[Theorem 2.3.1]{BC22}, base change to the $I_0$-adically completion of $R$ changes $\mathrm{Idem}(U)$, so we assume that every $R_i$ is $I_0$-adically complete (in particular, $I_0$-adically separated).
By \cref{thm:bc}, Zariski descent and Beauville--Laszlo glueing \cite[Proposition 3.5]{Kun23}, we can reduce to the case $U=\Spec(R)$ or $U=\Spec(R)\setminus V(I_0)$.

The case $U=\Spec(R)$ (hence $U^{s.\flat}=\Spec(R^{s.\flat})$) can be found in \cite[Proposition 3.26]{HIS26}. To deal with the case $U=\Spec(R)\setminus V(I_0)$ (hence $U^{s.\flat}=\Spec(R^{s.\flat})\setminus V(I_0^{s.\flat})$), choose a generator $f_0\in R$ (resp.\ $f_0^{s.\flat}\in R^{s.\flat}$) of $I_0$ (resp.\ $I_0^{s.\flat}$). Then $A\coloneqq\wh{R_\infty}$ has tilt $A^\flat=\varprojlim_{x\mapsto x^p}(A/f_0A)$ (\cite[Lemma 3.2]{HIS26}). Thus we have the functorial, compatible, multiplicative isomorphism (\cite[(2.1.7.2)]{CS24})
\[
\varprojlim_{x\mapsto x^p}(A_{f_0}) \cong (A^\flat)_{f_0^{s.\flat}},
\]
where $A_{f_0}$ (resp.\ $(A^\flat)_{f_0^{s.\flat}}$) denotes the localization of $A$ (resp.\ $A^\flat$) by $f_0$ (resp.\ $f_0^{s.\flat}$). This induces a bijection
\begin{align*}
\mathrm{Idem}(A_{f_0}) &\cong \mathrm{Idem}((A^\flat)_{f_0^{s.\flat}}), \\
e &\mapsto (e,e,e,\ldots), \\
e_0 &\mapsfrom (e_0,e_1,e_2,\ldots).
\end{align*}
Since we assumed that $R_i$ is $I_0$-adically separated for any $i\geq 0$, so is $R_\infty$. Combined with \cref{cor:reduced}, we see that the canonical map $R\to \wh{R_\infty}=A$ is injective. On the other hand, since $R_i^{s.\flat}$ is $I_0$-adically complete for any $i\geq 0$ (\cite[Lemma 3.36]{INS25}), the canonical map $R^{s.\flat}\to \wh{R_\infty^{s.\flat}}\cong A^\flat$ is injective (cf.\ \cite[Lemma 3.55]{INS25}). Then one can check that the bijection above restricts to the desired one
\[
\mathrm{Idem}(R_{f_0}) \cong \mathrm{Idem}((R^{s.\flat})_{f_0^{s.\flat}}).
\]
\end{proof}

As for \'{e}tale cohomology, we have the following result.

\begin{proposition}
Let $\textrm{{\boldmath $R$}}=\{R_i,t_i\}_{i\geq 0}$ be a perfectoid tower arising from a pair $(R,I_0)$. Fix $i\geq 0$, and assume that both $R_i$ and $R_{i+1}$ are $I_0$-adically henselian. Then for any torsion abelian group $G$, there are identifications of \'{e}tale cohomology
\[
\begin{tikzcd}
\bbR\Gamma_{\et}(\Spec(R_i),G) \dar["\cong"'] & \bbR\Gamma_{\et}(\Spec(R_{i+1}),G) \dar["\cong"] \\
\bbR\Gamma_{\et}(\Spec(R_i/I_0R_i),G) & \bbR\Gamma_{\et}(\Spec(R_{i+1}/I_1R_{i+1}),G) \lar["\cong"']
\end{tikzcd}
\]
and
\[
\begin{tikzcd}
\bbR\Gamma_{\et}(\Spec(R_i),G) \dar["\cong"'] & \bbR\Gamma_{\et}(\Spec(R_i^{s.\flat}),G) \dar["\cong"] \\
\bbR\Gamma_{\et}(\Spec(R_i/I_0R_i),G) & \bbR\Gamma_{\et}(\Spec(R_i^{s.\flat}/I_0^{s.\flat}R_i^{s.\flat}),G) \lar["\cong"']
\end{tikzcd}
\]
\end{proposition}

\begin{proof}
We only have to show that the vertical arrows are isomorphisms, which follows from affine analog of proper base change (\cite{Gab94}).
\end{proof}

For a perfectoid ring $A$ that is $\varpi$-adically complete for a $\varpi\in A$ with $p\in\varpi^pA$, K.~\v{C}esnavicius and P.~Scholze (\cite[Theorem 2.2.7]{CS24}; see also \cite[Theorem 3.6]{Kun23}) gave an algebraic proof of the isomorphism
\[
\bbR\Gamma_{\et}(U,G) \cong \bbR\Gamma_{\et}(U^\flat,G),
\]
where $\Spec(A)\setminus V(\varpi A) \subset U\subset\Spec(A)$ and $\Spec(A^\flat)\setminus V(\varpi^\flat A^\flat)\subset U^\flat\subset\Spec(A^\flat)$ are open subsets corresponding to each other. Although their proof is based on the general result that \'{e}tale cohomology satisfies arc hyperdescent (\cite[Theorem 2.2.5]{CS24}), it cannot work for perfectoid towers. Indeed, given a perfectoid tower $\textrm{{\boldmath $R$}}=\{R_i,t_i\}_{i\geq 0}$ arising from a pair $(R,I_0)$ such that every $t_i$ is integral injective, we can construct, as in the proof of \cite[Theorem 2.2.7]{CS24}, an $I_0$-complete arc hypercover
\[
\begin{tikzcd}
R\rar & \big( A_0 \rar[shift left] \rar[shift right] & A_1 \rar[shift left] \rar \rar[shift right] & \cdots \big),
\end{tikzcd}
\]
where $R\to A_0$ factors through $R\to\wh{R_\infty}$ and $A_n=\prod_{\lambda\in\Lambda_n}V_\lambda$ for $I_0$-adically complete valuation rings $V_\lambda$ over $\wh{R_\infty}$ with algebraically closed fraction fields (\cite[Lemma 2.2.3]{CS24} and \cite[\href{https://stacks.math.columbia.edu/tag/0DAV}{Tag 0DAV}]{stacks-project}). Then we hope that
\[
\begin{tikzcd}
R^{s.\flat}\rar & \big( A_0^\flat \rar[shift left] \rar[shift right] & A_1^\flat \rar[shift left] \rar \rar[shift right] & \cdots \big)
\end{tikzcd}
\]
is an $I_0^{s.\flat}$-complete arc hypercover, that is, for every $n\geq 0$ the canonical morphism $A^\flat_{n+1}\to (\cosk_n(\sk_n(A^\flat_\bullet)))_{n+1}$ is an $I_0^{s.\flat}$-complete arc cover. If we work on the category of $\wh{R_\infty}$-algebras, then $(\cosk_n(\sk_n(A_\bullet)))_{n+1}$ is an $I_0$-adically complete perfectoid ring with tilt $(\cosk_n(\sk_n(A^\flat_\bullet)))_{n+1}$, and so our claim follows from \cite[Lemma 2.2.2]{CS24}. However, we need work on the category of $R$-algebras, and thus $(\cosk_n(\sk_n(A_\bullet)))_{n+1}$ is not necessarily perfectoid. This is why an argument of \v{C}esnavi\v{c}ius--Scholze does not fit in the framework of pefectoid towers.

Here let us give one example of tilting \'{e}tale cohomology of perfectoid towers. Consider the perfectoid tower $\{\Z_p[p^{1/p^i}]\}_{i\geq 0}$ arising from $(\Z_p,p)$ whose tilt is $\{\F_p\llbracket T^{1/p^i}\rrbracket\}_{i\geq 0}$, and the open sets
\[
U=\Spec(\Z_p)\setminus V(p)=\Spec(\Q_p),\quad U^{s.\flat}=\Spec(\F_p\llbracket T\rrbracket)\setminus V(T)=\Spec \F_p((T))
\]
corresponding to each other (see \cite[Definition 4.3]{INS25}). By local class field theory, we observe that
\[
H^1_{\et}(U,\Z/p\Z) = H^1(G_{\Q_p},\Z/p\Z) = \Hom_{\mathrm{cont}}(G_{\Q_p},\Z/p\Z) = \Hom_{\mathrm{cont}}(\Q_p^\times,\Z/p\Z),
\]
which is a $2$-dimensional $\F_p$-vector space. On the other hand, by a similar argument (or, by using the Artin--Schreier sequence), $H^1_{\et}(U^{s.\flat},\Z/p\Z)$ has infinite dimension over $\F_p$.

%%%%%%%%%%%%%%%%%%%%%%%%%%%%%%%%%%%%%%%%%%%%%%%%%%%%%%%%%%%%%%%%%%%%%%%%%%%%%%%%%%%%%%%%%%%%%%%
\subsection{Tilting Koszul homology}
\label{ss:Koszul}

The second tilting invariant is Koszul homology $H_q(x_1,\ldots,x_n;R)$, which was already considered in \cite[Remark 3.40]{INS25} when $n=1$ and $x_1$ is a generator of $I_0$. Let us first remark the following canonical morphism on Koszul complexes.

\begin{definition}
Let $R$ be a ring, $\bm{x}=x_1,\ldots,x_n$ a sequence of elements in $R$, and $M$ an $R$-module. Let $K_\bullet(\bm{x};M)$ denote the Koszul complex of $\bm{x}$ with coefficients in $M$. Set $\bm{x}'\coloneqq x_2,\ldots,x_n$. We denote by
\[
\pi_{\bm{x},M}\colon K_\bullet(\bm{x};M) \to K_\bullet(\bm{x}';M/x_1M)
\]
the canonical morphism of $R$-complexes
\[
K_\bullet(\bm{x};M) = K_\bullet(x_1;M)\otimes_R K_\bullet(\bm{x}') \to (M/x_1M)\otimes_RK_\bullet(\bm{x}') = K_\bullet(\bm{x}';M/x_1M).
\]
\end{definition}

\begin{remark}
\label{rem:pi}
  \begin{enumerate}
  \item $\pi_{\bm{x},M}$ is functorial with respect to $M$.
  \item If $x_1$ is $M$-regular, then $\pi_{\bm{x},M}$ is a quasi-isomorphism.
  \end{enumerate}
\end{remark}

\begin{lemma}
\label{lem:Kos}
Let $R$ be a ring. Let $(A,I)$ and $(B,J)$ be pairs satisfying the following conditions.
  \begin{itemize}
  \item $A$ and $B$ are $R$-algebras.
  \item $I$ and $J$ are principal.
  \item $IA_{\Itor}=(0)$ and $JB_{\textrm{$J$-}\mathrm{tor}}=(0)$.
  \item There exists a commutative diagram of \emph{(}possibly non-unital\emph{)} $R$-algebras
  \[
  \begin{tikzcd}
  A_{\Itor} \rar["\varphi_{I,A}"] \dar["\Phi_{\mathrm{tor}}"',"\cong"] & A/I \dar["\Phi","\cong"'] \\
  B_{\textrm{$J$-}\mathrm{tor}} \rar["\varphi_{J,B}"'] & B/J.
  \end{tikzcd}
  \]
  \end{itemize}
Let $\bm{x}=x_1,\ldots,x_n$ and $\bm{y}=y_1,\ldots,y_n$ be sequences of elements in $A$ and $B$, respectively, such that $I=x_1A$, $J=y_1B$, and $\Phi(x_k\ \mathrm{mod}\ I)=y_k\ \mathrm{mod}\ J$ for all $1\leq k\leq n$. Then there exists an isomorphism
\[
K_\bullet(\bm{x};A) \cong K_\bullet(\bm{y};B)
\]
in the derived category $D(R)$ of $R$-modules.
\end{lemma}

\begin{proof}

Let $\wt{A}\coloneqq A/A_{\Itor}$ and $\wt{B}\coloneqq B/B_{\Jtor}$. Then we have decompositions
\begin{equation}
A \xr{\cong} \wt{A}\times_{\wt{A}/I\wt{A}} (A/I),\quad B\xr{\cong} \wt{B}\times_{\wt{B}/J\wt{B}}(B/J)
\end{equation}
by \cref{lem:BC}. From this we obtain the exact sequences of $R$-modules
\[
0\to A \to \wt{A}\oplus (A/I) \to \wt{A}/I\wt{A} \to 0, \quad 0\to B \to \wt{B}\oplus(B/J) \to \wt{B}/J\wt{B}\to 0.
\]
Set $\bm{x}'\coloneqq x_2,\ldots,x_n$ and $\bm{y}'\coloneqq y_2,\ldots,y_n$. 
Since the components of the Koszul complex are (finitely generated) free, we have the commutative diagram of complexes of $R$-modules with exact rows
\[
\begin{tikzcd}
0 \rar & K_\bullet(\bm{x};A) \rar \dar["\pi_{\bm{x},A}"'] & K_\bullet(\bm{x};\wt{A})\oplus K_\bullet(\bm{x};A/I) \rar \dar["{\pi_{\bm{x},\wt{A}}\oplus\pi_{\bm{x},A/I}}"'] & K_\bullet(\bm{x};\wt{A}/I\wt{A}) \rar \dar["\pi_{\bm{x},\wt{A}/I\wt{A}}"] & 0 \\
0 \rar & K_\bullet(\bm{x}';A/I) \rar \dar["\Phi"',"\cong"] & K_\bullet(\bm{x}';\wt{A}/I\wt{A})\oplus K_\bullet(\bm{x}';A/I) \rar \dar["\Phi"',"\cong"] & K_\bullet(\bm{x}';\wt{A}/I\wt{A}) \dar["\Phi","\cong"'] \rar & 0 \\
0 \rar & K_\bullet(\bm{y}';B/J) \rar & K_\bullet(\bm{y}';\wt{B}/J\wt{B})\oplus K_\bullet(\bm{y}';B/J) \rar & K_\bullet(\bm{y}';\wt{B}) \rar & 0 \\
0 \rar & K_\bullet(\bm{y};B) \rar \uar["\pi_{\bm{y},B}"] & K_\bullet(\bm{y},\wt{B})\oplus K_\bullet(\bm{y};B/J) \rar \uar["\pi_{\bm{y},\wt{B}}\oplus\pi_{\bm{y},B/J}"] & K_\bullet(\bm{y};\wt{B}/J\wt{B}) \rar \uar["\pi_{\bm{y},\wt{B}/J\wt{B}}"'] & 0
\end{tikzcd}
\]
Here $\pi_{\bm{x},\wt{A}}$ and $\pi_{\bm{y},\wt{B}}$ are quasi-isomorphisms by \cref{rem:pi} (2). Moreover, \cref{rem:pi} (1) establishes that, in $D(R)$, the zigzag
\[
K_\bullet(\bm{x};A/I) \xr{\pi_{\bm{x},A/I}} K_\bullet(\bm{x}';A/I) \xr{\Phi} K_\bullet(\bm{y}';B/J) \xleftarrow{\pi_{\bm{y},B/J}} K_\bullet(\bm{y};B/J)
\]
coincides with the isomorphism $K_\bullet(\bm{x};A/I) \xr{\cong} K_\bullet(\bm{y};B/J)$ induced by $\Phi$. The similar assertion holds for $\wt{A}/I\wt{A}$ and $\wt{B}/J\wt{B}$ because $\Phi$ induces a ring isomorphism $\wt{\Phi}\colon\wt{A}/I\wt{A}\xr{\cong} \wt{B}/J\wt{B}$.\footnote{Note that the maps $\varphi_{I,A}$ and $\varphi_{J,B}$ are injective by \cite[Corollary 3.15]{INS25}.} Hence we have the desired isomorphism of distinguished triangles in $D(R)$:
\[
\begin{tikzcd}
K_\bullet(\bm{x};A) \rar \dar["\cong","{}^\exists"'] & K_\bullet(\bm{x};\wt{A})\oplus K_\bullet(\bm{x};A/I) \rar \dar["\cong"] & K_\bullet(\bm{x};\wt{A}/I\wt{A}) \rar[dashed] \dar["\cong"] & {} \\
K_\bullet(\bm{y};B) \rar & K_\bullet(\bm{y},\wt{B})\oplus K_\bullet(\bm{y};B/J) \rar & K_\bullet(\bm{y};\wt{B}/J\wt{B}) \rar[dashed] & {}.
\end{tikzcd}
\]
\end{proof}

The following remark is based on a private communication with Ryo Ishizuka.

\begin{remark}
If we regard Koszul complexes $K(\bm{x};A)$ as \emph{derived quotients}
\[
A/^\bbL\bm{x}=A\otimes_{\Z[X_1,\ldots,X_n]}^\bbL\Z
\]
(where $\Z[X_1,\ldots,X_n]\to A$ is defined by $X_k\mapsto x_k$), then we can prove the slightly more precise statement of \cref{lem:Kos}: $A/^\bbL\bm{x}\cong B/^\bbL\bm{y}$ as animated rings.
Indeed, by assumption, we are given an isomorphism
\[
A/^\bbL x_1 \cong B/^\bbL y_1
\]
as commutative algebra objects in the derived $\infty$-category $\calD(\Z)=\calD(\Ab)$ of abelian groups. We endow $A/^\bbL x_1$ (resp.\ $B/^\bbL y_1$) with the $\Z[X_2,\ldots,X_n]$-algebra structure by setting $X_k\mapsto x_k$ (resp.\ $X_k\mapsto y_k$) for all $2\leq k\leq n$. Then the isomorphism above is that in $\calD(\Z[X_2,\ldots,X_n])$, and hence, applying $-\otimes_{\Z[X_2,\ldots,X_n]}^\bbL\Z=-\otimes_{\Z[X_2,\ldots,X_n]}^\bbL\Z[X_2,\ldots,X_n]/(X_2,\ldots,X_n)$, we get the desired isomorphism
\begin{align*}
A/^\bbL\bm{x} \cong (A/^\bbL x_1)/^\bbL(x_2,\ldots,x_n) &\cong A/^\bbL x_1 \otimes_{\Z[X_2,\ldots,X_n]}^\bbL \Z \\
&\cong B/^\bbL y_1 \otimes_{\Z[X_2,\ldots,X_n]}^\bbL \Z \cong (B/^\bbL y_1)/^\bbL(y_2,\ldots,y_n) \cong B/^\bbL\bm{y}.
\end{align*}
\end{remark}

Now our main theorem of this subsection is established as follows.

\begin{theorem}
\label{thm:tilt-Kos}
Let $\textrm{{\boldmath $R$}}=\{R_i,t_i\}_{i\geq 0}$ be a perfectoid tower arising from a pair $(R,I_0)$. Fix $i\geq 0$. Let $\bm{x}=x_1,\ldots,x_n$ and $\bm{x}^{s.\flat}=x_1^{s.\flat},\ldots,x_n^{s.\flat}$ be sequences of elements in $R_i$ and $R_i^{s.\flat}$, respectively, such that $I_0R_i=(x_1)$, $I_0^{s.\flat}R_i^{s.\flat}=(x_1^{s.\flat})$, and $\Phi_0^{(i)}(x_k^{s.\flat})=x_k\ \mathrm{mod}\ I_0R_i$ for all $1\leq k\leq n$. Then there exists an isomorphism
\[
K_\bullet(\bm{x};R_i) \cong K_\bullet(\bm{x}^{s.\flat};R_i^{s.\flat})
\]
in the derived category $D(\Ab)$ of abelian groups.
\end{theorem}

\begin{proof}
Due to \cref{thm:INS3.35} (1), we may apply \cref{lem:Kos} to the pairs $(R_i^{s.\flat},I_0^{s.\flat}R_i^{s.\flat})$ and $(R_i,I_0R_i)$.
\end{proof}

%%%%%%%%%%%%%%%%%%%%%%%%%%%%%%%%%%%%%%%%%%%%%%%%%%%%%%%%%%%%%%%%%%%%%%%%%%%%%%%%%%%%%%%%%%%%%%%
\subsection{Tilting several invariants of Noetherian local rings}

In this subsection, we focus on Noetherian local perfectoid towers, that is, perfectoid towers consisting of Noetherian local rings (\cref{def:towerP}). Note that if a perfectoid tower {\boldmath $R$} arising from a pair $(R,I_0)$ is Noetherian local, then so is the tilt $\textrm{{\boldmath $R$}}^\flat$ (\cite[Lemma 3.11 (2), Proposition 3.42 (2)]{INS25}).

We begin with the tilting invariance of Krull dimension, which was already considered in \cite[Proposition 3.42 (3)]{INS25} in the $I_0$-torsion free case. The general case can be proven with the aid of \cref{thm:decomp}:

\begin{proposition}
\label{prop:tilt-dim}
Let $\textrm{{\boldmath $R$}}=\{R_i,t_i\}_{i\geq 0}$ be a Noetherian local perfectoid tower arising from a pair $(R,I_0)$. Then we have the equalities
\[
\begin{tikzcd}[sep=small]
\dim(R) \rar[equal] \dar[equal] & \dim(R_0) \rar[equal] \dar[equal] & \dim(R_1) \rar[equal] \dar[equal] & \cdots \rar[equal] & \dim(R_i) \rar[equal] \dar[equal] & \cdots \\
\dim(R^{s.\flat}) \rar[equal] & \dim(R_0^{s.\flat}) \rar[equal] & \dim(R_1^{s.\flat}) \rar[equal] & \cdots \rar[equal] & \dim(R_i^{s.\flat}) \rar[equal] & \cdots.
\end{tikzcd}
\]
\end{proposition}

\begin{proof}
Fix $i\geq 0$, and let $\wt{R_i}\coloneqq R_i/(R_i)_{\Iztor}$ and $\wt{R_i}^{s.\flat}\coloneqq R_i^{s.\flat}/(R_i^{s.\flat})_{\Izstor}$. Then we have decompositions
\[
R_i\xr{\cong} \wt{R_i}\times_{\wt{R_i}/I_0\wt{R_i}}(R_i/I_0R_i),\quad R_i^{s.\flat}\xr{\cong}\wt{R_i}^{s.\flat}\times_{\wt{R_i}/I_0\wt{R_i}}(R_i/I_0R_i),
\]
which yields
\[
\dim(R_i) = \max\{\dim(\wt{R_i}), \dim(R_i/I_0R_i)\},\quad \dim(R_i^{s.\flat}) = \max\{\dim(\wt{R_i}^{s.\flat}), \dim(R_i/I_0R_i)\}
\]
(cf.\ \cite[Corollary 3.2]{BF26}). Hence it suffices to show that $\dim \wt{R_i}=\dim\wt{R_i}^{s.\flat}$, and thus we may assume that {\boldmath $R$} is $I_0$-torsion free (due to \cref{prop:torftower} (2) (3) (4)). Then the claim follows from the isomorphism $\wt{R_i}^{s.\flat}/I_0^{s.\flat}\wt{R_i}^{s.\flat} \xr{\cong} \wt{R_i}/I_0\wt{R_i}$.
By a similar argument one can also show the equality $\dim(R_i)=\dim(R_{i+1})$. Alternatively, it is enough to invoke that $R^{s.\flat}\cong R_1^{s.\flat}\cong\cdots\cong R_i^{s.\flat}\cong\cdots$ as rings (\cref{ex:perfectoid tower} (1)).
\end{proof}

For a Noetherian local ring $(R,\m,k)$, let $\edim(R)=\dim_k(\m/\m^2)$ denote the \emph{embedding dimension of $R$}. Let us here make the following observation.

\begin{lemma}
\label{lem:mIm}
Let $\textrm{{\boldmath $R$}}=\{R_i\}_{i\geq 0}$ be a local perfectoid tower arising from a pair $(R,I_0)$. For each $i\geq 0$, let $\m_i$ and $\m_i^{s.\flat}$ denote the maximal ideals of $R_i$ and $R_i^{s.\flat}$, respectively. Then we have the canonical isomorphisms
\[
\begin{tikzcd}[sep=small]
 & \m_1/I_1\m_1 & \cdots \lar["\cong"'] & \m_i/I_i\m_i \lar["\cong"'] & \cdots \lar["\cong"'] \\
\m_0^{s.\flat}/I_0^{s.\flat}\m_0^{s.\flat} & \m_1^{s.\flat}/I_1^{s.\flat}\m_1^{s.\flat} \lar["\cong"'] \uar["\cong"] & \cdots \lar["\cong"'] & \m_i^{s.\flat}/I_i^{s.\flat}\m_i^{s.\flat} \lar["\cong"'] \uar["\cong"'] & \cdots \lar["\cong"'].
\end{tikzcd}
\]
\end{lemma}

\begin{proof}
Fix $i\geq 0$, and choose a generator $f_{i+1}^{s.\flat}$ of $I_{i+1}^{s.\flat}$. Then for any $j\geq 0$ the image $\ol{f_{i+j+1}}\in R_{i+j+1}/I_0R_{i+j+1}$ of $f_{i+1}^{s.\flat}$ is a generator of $I_{i+j+1}/I_0R_{i+j+1}$ by \cite[Theorem 3.35 (1)]{INS25}. Hence we have the commutative diagram with exact rows
\[
\begin{tikzcd}
\m_{i+2}/I_1R_{i+2} \rar["\ol{f_{i+2}}"] \dar["\cong"] & \m_{i+2}/I_1R_{i+2} \rar \dar["\cong"] & \m_{i+2}/I_{i+2}\m_{i+2} \rar \dar[dashed,"\cong"] & 0 \\
\m_{i+1}/I_0R_{i+j+1} \rar["\ol{f_{i+1}}"] & \m_{i+1}/I_0R_{i+1} \rar & \m_{i+1}/I_{i+1}\m_{i+1} \rar & 0,
\end{tikzcd}
\]
where the left two vertical arrows are induced by the isomorphism of local rings $R_{i+2}/I_1R_{i+2}\xr{\cong} R_{i+1}/I_0R_{i+1}$. Similarly, we have the commutative diagram with exact rows
\[
\begin{tikzcd}
\m_{i+1}^{s.\flat}/I_0^{s.\flat}R_{i+1}^{s.\flat} \rar["f_{i+1}^{s.\flat}"] \dar["\cong"] & \m_{i+1}^{s.\flat}/I_0^{s.\flat}R_{i+1}^{s.\flat} \rar \dar["\cong"] & \m_{i+1}^{s.\flat}/I_{i+1}^{s.\flat}\m_{i+1}^{s.\flat} \rar \dar[dashed,"\cong"] & 0 \\
\m_{i+1}/I_0R_{i+1} \rar["\ol{f_{i+1}}"] & \m_{i+1}/I_0R_{i+1} \rar & \m_{i+1}/I_{i+1}\m_{i+1}\rar & 0,
\end{tikzcd}
\]
where the left two vertical arrows are induced by the isomorphism of local rings $R_{i+1}^{s.\flat}/I_0^{s.\flat}R_{i+1}^{s.\flat}\xr{\cong}R_{i+1}/I_0R_{i+1}$.
\end{proof}

If $(R,\m)$ is a Noetherian local ring and $I\subset\m$ is a principal ideal, then
\[
\edim(R)=
\begin{cases}
\edim(R/I)+1 & (I\not\subset\m^2), \\
\edim(R/I) & (I\subset\m^2).
\end{cases}
\]
Thus the following result on local perfectoid towers will be useful.

\begin{lemma}
\label{lem:m2}
Let $\textrm{{\boldmath $R$}}=\{R_i,t_i\}_{i\geq 0}$ be a local perfectoid tower arising from a pair $(R,I_0)$. Assume that the transition maps $t_i\colon R_i\to R_{i+1}$ are injective. Then for any $i\geq 0$, $I_i\subset\m_i^2$ if and only if $I_{i+1}\subset\m_{i+1}^2$. Hence $I_0\subset\m^2$ if and only if $I_i\subset\m_i^2$ for all $i\geq 0$.
\end{lemma}

It was proved by Shimomoto \cite[Proposition 4.9]{Shi16} that a ramified complete regular local ring admits a faithfully flat perfectoid algebra; the following argument is essentially the same as the one therein.

\begin{proof}
By shifting, we may assume $i=0$. If $I_0=(0)$, then {\boldmath $R$} is a perfect tower, and thus the assertion is obvious. We henceforth assume that $I_0\neq(0)$. Then $I_0R_1\neq(0)$ since the map $t_0\colon R_0\to R_1$ was assumed to be injective.

First, assume that $I_0\subset\m_0^2$. Choose a generator $f_0$ of $I_0$. Then we may write $f_0=\sum_{i=1}^nb_ib'_i$ with $b_i,b'_i\in\m_0$. Since the Frobenius projection $F_0\colon R_1/I_0R_1\to R_0/I_0$ is surjective, we may find $c_i,c'_i\in\m_1$ and $d_i,d'_i\in R_1$ such that
\[
b_i=c_i^p+f_0d_i,\quad b'_i={c'_i}^p+f_0d'_i.
\]
Substituting these equations into $f_0=\sum_{i=1}^n b_i b'_i$, we obtain $f_0(1-h_1)=\sum_{i=1}^n c_i^p {c'_i}^p$ with the property that $1-h_1\in R_1^\times$. Set
\[
f_1\coloneqq \sum_{i=1}^n c_ic'_i \in \m_1^2,
\]
and since $p\in I_0$, we can write
\[
\sum_{i=1}^nc_i^p{c'_i}^p - f_1^p=\sum_{i=1}^nc_i^p{c'_i}^p - \left(\sum_{i=1}^nc_ic'_i\right)^p
=f_0g_1
\]
for some $g_1\in\m_1$. Hence $f_0(1-h_1)=f_1^p+f_0g_1$ and $f_0(1-g_1-h_1)=f_1^p$. So $f_1^pR_1=I_0R_1=I_1^p$. Taking radicals, we get $\m_1^2\supset f_1R_1=I_1$.

Next, assume that $I_1\subset\m_1^2$. Choose a generator $f_1$ of $I_1$. Then we may write $f_1=\sum_{i=1}^nc_ic'_i$ with $c_i,c'_i\in\m_1$. Since the absolute Frobenius on $R_1/I_0R_1$ factors through the Frobenius projection $F_0\colon R_1/I_0R_1\to R_0/I_0$, we may find $b_i,b'_i\in\m_0$ and $d_i,d'_i\in R_1$ such that
\[
b_i=c_i^p+f_1^pd_i,\quad b'_i={c'_i}^p+f_1^pd'_i.
\]
Set
\[
g_0\coloneqq \sum_{i=1}^n b_ib'_i \in \m_0^2,
\]
and since $p\in I_0R_1=I_1^p$, we can write
\[
\sum_{i=1}^nc_i^p{c'_i}^p - f_1^p = \sum_{i=1}^nc_i^p{c'_i}^p - \left(\sum_{i=1}^nc_ic'_i\right)^p = f_1^p g_1
\]
for some $g_1\in\m_1$. Hence we obtain
\begin{align*}
f_1^p &= \sum_{i=1}^n c_i^p {c'_i}^p - \left(\sum_{i=1}^n c_i^p {c'_i}^p - \left(\sum_{i=1}^n c_ic'_i\right)^p \right) \\
&= \sum_{i=1}^n (b_i-f_1^pd_i)(b'_i-f_1^p d'_i) - f_1^pg_1 \\
&= g_0 +f_1^ph_1-f_1^pg_1,
\end{align*}
where $h_1\in \m_1$. Thus $(1+g_1-h_1)f_1^p=g_0\in I_0R_1\cap \m^2\subset I_0$. Set $u_1\coloneqq 1+g_1-h_1\in R_1^\times$. Hence $u_1f_1^p=a_0f_0$ for some $a_0\in R_0$ and a generator $f_0$ of $I_0$. We may write $f_0=a_1f_1$ with $a_1\in R_1$. Then $(u_1-a_0a_1^p)f_1^p=0$. If $a_0\in\m_0$, then we would have $u_1-a_0a_1^p=u_1(1-u_1^{-1}a_0a_1^p)\in R_1^\times$, and thus $f_1^p=0$, which contradicts the assumption that $I_0R_1\neq(0)$. Hence $a_0\in R^\times$, and so $I_0=(a_0f_0)=(g_0)\subset\m_0^2$.
\end{proof}

Using the lemmas above, we readily obtain the following theorem.

%\begin{lemma}
%\label{lem:torZar}
%Let $\textrm{{\boldmath $R$}}=\{R_i,t_i\}_{i\geq 0}$ be a local perfectoid tower arising from a pair $(R,I_0)$. Then for any $i\geq 0$ the following conditions are equivalent.
%  \begin{enumerate}
%  \item $I_0R_i\neq(0)$.
%  \item $(R_i)_{\Iztor}$ is contained in the maximal ideal of $R_i$.
%  \item $(R_i^{s.\flat})_{\Izstor}$ is contained in the maximal ideal of $R_i^{s.\flat}$.
%  \end{enumerate}
%\end{lemma}
%
%\begin{proof}
%(1) $\Leftrightarrow$ (2): $(R_i)_{\Iztor}$ contains a unit if and only if $I_0R_i=(0)$.
%
%(2) $\Leftrightarrow$ (3): By \cref{thm:INS3.35} (1), $R_i/I_0R_i$ is $\Im(\varphi_{I_0,R_i})$-adically Zariskian if and only if $R_i^{s.\flat}/I_0^{s.\flat}R_i^{s.\flat}$ is $\Im(\varphi_{I_0^{s.\flat},R_i^{s.\flat}})$-adically Zariskian. The assertion follows from this.
%\end{proof}

\begin{theorem}
\label{thm:tilt-regular}
Let $\textrm{{\boldmath $R$}}=\{R_i,t_i\}_{i\geq 0}$ be a Noetherian local perfectoid tower arising from a pair $(R,I_0)$.
  \begin{enumerate}
  \item We have the equalities
  \[
  \begin{tikzcd}[sep=small]
  \edim(R) \rar[equal] \dar[equal] & \edim(R_0) \rar[equal] \dar[equal] & \edim(R_1) \rar[equal] \dar[equal] & \cdots \rar[equal] & \edim(R_i) \rar[equal] \dar[equal] & \cdots \\
  \edim(R^{s.\flat}) \rar[equal] & \edim(R_0^{s.\flat}) \rar[equal] & \edim(R_1^{s.\flat}) \rar[equal] & \cdots \rar[equal] & \edim(R_i^{s.\flat}) \rar[equal] & \cdots.
  \end{tikzcd}
  \]
  \item For each $i\geq 0$, the following conditions are equivalent.
    \begin{enumerate}
    \item $I_i\subset\m_i^2$, where $\m_i$ is the maximal ideal of $R_i$.
    \item $I_i^{s.\flat}\subset(\m_i^{s.\flat})^2$, where $\m_i^{s.\flat}$ is the maximal ideal of $R_i^{s.\flat}$.
    \end{enumerate}
  Moreover, if these equivalent conditions are satisfied for some $i\geq 0$, then they are satisfied for every $i\geq 0$.
  \item For each $i\geq 0$, the following conditions are equivalent.
    \begin{enumerate}
    \item $R_i$ is regular.
    \item $R_i^{s.\flat}$ is regular.
    \end{enumerate}
  Moreover, if these equivalent conditions are satisfied for some $i\geq 0$, then they are satisfied for every $i\geq 0$.
  \end{enumerate}
\end{theorem}

\begin{proof}
(1) For every $i\geq 0$, it follows from \cref{lem:mIm} that 
\begin{align*}
\edim(R_{i+1})&= \dim_k((\m_{i+1}/I_{i+1}\m_{i+1}) \otimes_{R_{i+1}/I_{i+1}}k) \\
&= \dim_k((\m_{i+1}^{s.\flat}/I_{i+1}^{s.\flat}\m_{i+1}^{s.\flat})\otimes_{R_{i+1}^{s.\flat}/I_{i+1}^{s.\flat}} k) = \edim(R_{i+1}^{s.\flat}).
\end{align*}
It remains to show that $\edim(R)=\edim(R_1)$, which follows from \cref{lem:m2}:
\begin{align*}
\edim(R) &=
\begin{cases}
\edim(R/I_0)+1 & (I_0\not\subset\m^2), \\
\edim(R/I_0) & (I_0\subset\m^2)
\end{cases} \\
&=
\begin{cases}
\edim(R_1/I_1)+1 & (I_1\not\subset\m_1^2), \\
\edim(R_1/I_1) & (I_1\subset\m_1^2)
\end{cases} \\
&=\edim(R_1).
\end{align*}

(2) follows from \cref{lem:m2}, (1), and the isomorphism $R_i^{s.\flat}/I_i^{s.\flat}\xr{\cong}R_i/I_i$.

(3) follows from \cref{prop:tilt-dim} and (1).
\end{proof}

Note that \cite{Ha26c} gives another proof of \cref{thm:tilt-regular} (3) using a mixed characteristic analogue of Kunz's theorem established by O.~Gabber and J.~Lurie.

For a Noetherian local ring $R$, let $\depth(R)$ denote the depth. As usual, we say that $R$ is \emph{Cohen--Macaulay} if $\depth(R)=\dim(R)$. The tilting correspondence of Koszul homology yields the following result.

\begin{theorem}
\label{thm:tilt-CM}
Let $\textrm{{\boldmath $R$}}=\{R_i,t_i\}_{i\geq 0}$ be a Noetherian local perfectoid tower arising from a pair $(R,I_0)$.
  \begin{enumerate}
  \item We have the equalities
  \[
  \begin{tikzcd}[sep=small]
  \depth(R) \rar[equal] \dar[equal] & \depth(R_0) \rar[equal] \dar[equal] & \depth(R_1) \rar[equal] \dar[equal] & \cdots \rar[equal] & \depth(R_i) \rar[equal] \dar[equal] & \cdots \\
  \depth(R^{s.\flat}) \rar[equal] & \depth(R_0^{s.\flat}) \rar[equal] & \depth(R_1^{s.\flat}) \rar[equal] & \cdots \rar[equal] & \depth(R_i^{s.\flat}) \rar[equal] & \cdots.
  \end{tikzcd}
  \]
  \item For each $i\geq 0$, the following conditions are equivalent.
    \begin{enumerate}
    \item $R_i$ is Cohen--Macaulay \emph{(}resp.\ hypersurface\footnote{We say that a Noetherian local ring $R$ is \emph{hypersurface} if $\edim(R)-\depth(R)\leq 1$ (cf.\ \cite[5.1]{Avr}).}\emph{)}.
    \item $R_i^{s.\flat}$ is Cohen--Macaulay \emph{(}resp.\ hypersurface\emph{)}.
    \end{enumerate}
  Moreover, if these equivalent conditions are satisfied for some $i\geq 0$, then they are satisfied for every $i\geq 0$.
  \end{enumerate}
\end{theorem}

\begin{proof}
(1) follows from \cref{thm:tilt-Kos} and the depth sensitivity \cite[Theorem 16.8]{Mat2}.

(2) follows from \cref{prop:tilt-dim} (resp.\ \cref{thm:tilt-regular} (1)) and (1).
\end{proof}

Recall here that one can construct perfectoid towers arising from $p$-Stanley--Reisner rings (\cref{ex:perfectoid tower} (4)). From this theorem we obtain the following corollary.

\begin{corollary}
\label{cor:Reisner}
Let $\Delta$ be a simplicial complex, and $C(k)$ the Cohen ring of a field $k$ of characteristic $p>0$. The following conditions are equivalent.
  \begin{enumerate}
  \item The $p$-Stanley--Reisner ring $C(k)\ol{[\Delta]}$ is Cohen--Macaulay.
  \item The Stanley--Reisner ring $k[\Delta]$ is Cohen--Macaulay.
  \end{enumerate}
\end{corollary}

\begin{example}
Let $\Delta$ be a triangulation of the real projective plane $\R\mathbb{P}^2$. Then, by Reisner's criterion, one can check that $\F_p[\Delta]$ is Cohen--Macaulay if and only if $p\neq 2$ (\cite[Exercise 5.3.14]{BH}). Hence $\Z_p\ol{[\Delta]}$ is Cohen--Macaulay if and only if $p\neq 2$.
\end{example}

Let $(R,\m,k)$ be a Noetherian local ring. The (\emph{Cohen--Macaulay}) \emph{type} $r(R)$ is an invariant which refines the information given by the depth. For example, $R$ is Gorenstein if and only if it is Cohen--Macaulay of type $1$.
On the other hand, the Koszul complex $K_\bullet(\bm{x})$ of a minimal basis $\bm{x}$ of $\m$ is determined by $R$ up to isomorphism. This justifies the notation $K_\bullet(R)\coloneqq K_\bullet(\bm{x})$. For an integer $q$, the number $\varepsilon_q(R)\coloneqq \dim_kH_q(R)$ is called the \emph{$q$-th deviation of $R$}. As usual, we say that $R$ is \emph{complete intersection} if $\varepsilon_1(R)=\edim(R)-\dim(R)$.

If $R$ is Cohen--Macaulay, its type is given by the dimension of the top homology of $K_\bullet(R)$. This result is well-known to experts. As the author was not able to find it in some standard references or books, we decided to give a proof.

\begin{lemma}
\label{lem:typeKos}
Let $(R,\m,k)$ be a Cohen--Macaulay local ring.
  \begin{enumerate}
  \item $H_q(R)=0$ for all $q>\edim(R)-\dim(R)$.
  \item $r(R)=\varepsilon_{\edim(R)-\dim(R)}(R)=\dim_k H_{\edim(R)-\dim(R)}(R)$.
  \end{enumerate}
\end{lemma}

\begin{proof}
We may assume that $R$ is complete and of the form $R=S/\fraka$, where $(S,\n,k)$ is a regular local ring with $\fraka\subset\n^2$. Choose a regular system of parameters (that is a minimal basis of $\n$) $\bm{y}=y_1,\ldots,y_e$. Then the images $x_i$ of $y_i$ form a minimal basis $\bm{x}=x_1,\ldots,x_e$ of $\m$. Note that $R\otimes_S^\bbL k = R\otimes_S K_\bullet(\bm{y}) = K_\bullet(\bm{x})$.

(1) By the Auslander--Buchsbaum formula, the projective dimension of $R$ over $S$ is given by
\[
g\coloneqq \pd_S(R)=\dim(S)-\dim(R)=\edim(R)-\dim(R).
\]
The assertion follows from this.

(2) Let $F_\bullet=(F_g\to F_{g-1}\to\cdots\to F_0)$ be a minimal free $S$-resolution of $R$. Then the $S$-duals $G_q\coloneqq\Hom_S(F_{g-q},S)$ of $F_{g-q}$ form a minimal free $S$-resolution $G_\bullet=(G_g\to G_{g-1}\to\cdots\to G_0)$ of the canonical module $\omega_R$ (\cite[Corollary 3.3.9]{BH}). Hence
\[
r_R(R)=\mu_R(\omega_R)=\mu_S(\omega_R)=\dim_k(\omega_R\otimes_Sk)=\rk_S(G_0) = \rk_S(F_g) = \dim_k\Tor^S_g(R,k).
\]
(cf.\ \cite[Propositions 3.3.11 (c) (ii) and 1.3.1 (c)]{BH}). This confirms the assertion.
\end{proof}

Then one can readily establish the following theorem.

\begin{theorem}
\label{thm:tilt-Gor}
Let $\textrm{{\boldmath $R$}}=\{R_i,t_i\}_{i\geq 0}$ be a Noetherian local perfectoid tower arising from a pair $(R,I_0)$. Assume that the residue field of $R$ is perfect.
  \begin{enumerate}
  \item For any $q\in\Z$, we have the equalities
  \[
  \begin{tikzcd}[sep=small]
  \varepsilon_q(R) \rar[equal] \dar[equal] & \varepsilon_q(R_0) \rar[equal] \dar[equal] & \varepsilon_q(R_1) \rar[equal] \dar[equal] & \cdots \rar[equal] & \varepsilon_q(R_i) \rar[equal] \dar[equal] & \cdots \\
  \varepsilon_q(R^{s.\flat}) \rar[equal] & \varepsilon_q(R_0^{s.\flat}) \rar[equal] & \varepsilon_q(R_1^{s.\flat}) \rar[equal] & \cdots \rar[equal] & \varepsilon_q(R_i^{s.\flat}) \rar[equal] & \cdots.
  \end{tikzcd}
  \]
  \item Let $r$ be a positive integer. For each $i\geq 0$, the following conditions are equivalent.
    \begin{enumerate}
    \item $R_i$ is Cohen--Macaulay of type $r$ \emph{(}resp.\ Gorenstein, complete intersection\emph{)}.
    \item $R_i^{s.\flat}$ is Cohen--Macaulay of type $r$ \emph{(}resp.\ Gorenstein, complete intersection\emph{)}.
    \end{enumerate}
  Moreover, if these equivalent conditions are satisfied for some $i\geq 0$, then they are satisfied for every $i\geq 0$.
  \end{enumerate}
\end{theorem}

\begin{proof}
(1) Since the absolute Frobenius is compatible with completion and the completion of a Noetherian local ring is (faithfully) flat, the tower $\wh{R}=\wh{R_0}\to\wh{R_1}\to\cdots$ consisting of completions is again a perfectoid tower arising from $(\wh{R},I_0\wh{R})$. Hence we may assume that every local ring $R_i$ is complete.

We fix once for all a ring homomorphism $W(k)\to R$ lifting the identity map on the residue field $k$ of $R$. Then we can see that the diagram of $W(k)$-algebras
\[
\begin{tikzcd}
k \rar["\varphi"] \dar & k \dar \\
R_{i+1}/I_0R_{i+1} \rar["F_i"'] & R_i/I_0R_i
\end{tikzcd}
\]
commutes for any $i\geq 0$, and thus we get a homomorphism $k\to R^{s.\flat}$ because $k$ is perfect. Consequently, both $R$ and $R^{s.\flat}$ have the same base ring $W(k)$ as follows:
\[
\begin{tikzcd}
W(k) \rar \dar \ar[rrr,bend left] & R \rar & R/I_0 \rar & k \\
k \rar \ar[rrr,bend right,equal] \ar[rru] & R^{s.\flat} \rar \ar[ru] & R^{s.\flat}/I_0^{s.\flat} \uar["\cong"] \rar & k \uar[equal]
\end{tikzcd}
\]
Let $\m$ and $\m^{s.\flat}$ denote the maximal ideals of $R$ and $R^{s.\flat}$, respectively.

\textsc{Step 1}. The case $I_0\not\subset\m^2$. Then $I_0^{s.\flat}\not\subset(\m^{s.\flat})^2$ by \cref{thm:tilt-regular} (2). Thus we can take a minimal basis $\bm{x}=x_1,\ldots,x_n$ and $\bm{x}^{s.\flat}=x_1^{s.\flat},\ldots,x_n^{s.\flat}$ of $\m$ and $\m^{s.\flat}$, respectively, as in \cref{thm:tilt-Kos}. Then $K_\bullet(R)\cong K_\bullet(R^{s.\flat})$ in $D(W(k))$. Since the $q$-th homologies of the both sides are annihilated by $p$, we have $H_q(R)\cong H_q(R^{s.\flat})$ as $k$-vector spaces, as desired.

\textsc{Step 2}. The case $I_0\subset \m^2$. Then $I_0^{s.\flat}\subset(\m^{s.\flat})^2$ by \cref{thm:tilt-regular} (2). Thus minimal basis of $\m$ (resp.\ $\m^{s.\flat}$) is the same as that of $\m/I_0$ (resp.\ $\m^{s.\flat}/I_0^{s.\flat}$). Hence we can take a minimal basis $\bm{x}=x_1,\ldots,x_n$ and $\bm{x}^{s.\flat}=x_1^{s.\flat},\ldots,x_n^{s.\flat}$ of $\m$ and $\m^{s.\flat}$, respectively, such that $\Phi_0^{(0)}(x_j^{s.\flat})=x_j\ \mathrm{mod}\ I_0$ for all $1\leq j\leq n$. Choose a generator $f_0$ and $f_0^{s.\flat}$ of $I_0$ and $I_0^{s.\flat}$, respectively, and set $\bm{x}'\coloneqq f_0,x_1,\ldots,x_n$ and ${\bm{x}^{s.\flat}}'\coloneqq f_0^{s.\flat},x_1^{s.\flat},\ldots,x_n^{s.\flat}$. Then, as in \textsc{Step 1}, one has $H_q(\bm{x}')\cong H_q({\bm{x}^{s.\flat}}')$ as $k$-vector spaces for all $q\in\Z$. But
\[
H_q(\bm{x}')\cong H_q(\bm{x})\oplus H_{q-1}(\bm{x}),\quad H_q({\bm{x}^{s.\flat}}')\cong H_q(\bm{x}^{s.\flat})\oplus H_q(\bm{x}^{s.\flat})
\]
by \cite[Proposition 1.6.21]{BH}, and thus we get $\dim_kH_q(\bm{x})=\dim_kH_q(\bm{x}^{s.\flat})$ by induction on $q\geq 0$.

We therefore proved $\varepsilon_q(R)=\varepsilon_q(R^{s.\flat})$. By a similar argument one can also show the equality $\varepsilon_q(R_i)=\varepsilon_q(R_i^{s.\flat})$ and $\varepsilon_q(R_i)=\varepsilon_q(R_{i+1})$ for all $i\geq 0$.

(2) follows from \cref{thm:tilt-CM}, \cref{lem:typeKos} (2), \cref{prop:tilt-dim}, \cref{thm:tilt-regular} (1), and (1).
\end{proof}

%%%%%%%%%%%%%%%%%%%%%%%%%%%%%%%%%%%%%%%%%%%%%%%%%%%%%%%%%%%%%%%%%%%%%%%%%%%%%%%%%%%%%%%%%%%%%%%
%%%%%%%%%%%%%%%%%%%%%%%%%%%%%%%%%%%%%%%%%%%%%%%%%%%%%%%%%%%%%%%%%%%%%%%%%%%%%%%%%%%%%%%%%%%%%%%

\end{document}